\documentclass[graybox]{svmult}
\usepackage{mathptmx}       
\usepackage{helvet}         
\usepackage{courier}        
\usepackage{makeidx}         
\usepackage{graphicx}        
\usepackage{multicol}        
\usepackage[bottom]{footmisc}
\usepackage{amsmath,amssymb,amsfonts}        
\makeindex             

\usepackage{mathrsfs}
\usepackage{bbm}
\usepackage{bm}
\usepackage{subfigure}

\newcommand{\ignore}[1]{}

\allowdisplaybreaks 

\newcommand{\Pb}{{\mathsf{P}}} 
\newcommand{\EV}{{\mathsf{E}}} 
\newcommand{\Eb}{{\mathsf{E}}} 
\newcommand{\PFA}{{\mathsf{PFA}}} 
\newcommand{\ADD}{{\mathsf{EDD}}} 
\newcommand{\Hyp}{{\mathsf {H}}}

\newcommand{\Var}{\mathsf{Var}}
\newcommand{\B}{{\mathsf{B}}}

\newcommand{\Fc}{{ \mathscr{F}}} 

\newcommand{\Bc}{{ \mathscr{B}}}

\newcommand{\Pc}{{ \mathscr{P}}}

\newcommand{\mrm}[1]{\mathrm{#1}}
\newcommand{\drm}{{\mrm{d}}}
\newcommand{\F}{{\mrm{F}}}

\newcommand{\mb}[1]{\mathbf{#1}} 

\newcommand{\Yb}{\mb{Y}}

\newcommand{\Xb}{{\mb{X}}}

\newcommand{\pb}{{\mathbf{p}}}

\def\CP{{\mb{CP}}}

\def\e{{\mb e}}

\newcommand{\mbs}[1]{\bm{#1}} 

\newcommand{\teb}{{\mbs{\theta}}}

\newcommand{\mbb}[1]{\mathbb{#1}} 
\newcommand{\Rbb}{\mbb{R}} 
\newcommand{\Zbb}{\mbb{Z}} 

\newcommand{\ind}[1]{\mathbbm{1}_{\{#1\}}}     
\newcommand{\Ind}[1]{\mathbbm{1}_{\{#1\}}}     
\newcommand{\class}{{\mbb{C}}}
\newcommand{\classalpi}{{\mbb{C}_\pi}(\alpha)}

\def\bbr{{\mathbb R}}
\newcommand{\Nbb}{\mathbb{N}}

\newcommand{\mc}[1]{\mathcal{#1}} 

\newcommand{\Nc}{{\mc{N}}}
\newcommand{\Ac}{{\mc{A}}}
\newcommand{\Mc}{{\mc{M}}}
\newcommand{\Lc}{{\mc{L}}}

\newcommand{\Kc}{{\mc{K}}}

\newcommand{\wtT}{{\widetilde{T}}}

\newcommand{\wtX}{\widetilde{X}}
\newcommand{\wtS}{\widetilde{S}}

\newcommand{\wt}{\widetilde}

\newcommand{\xra}{\xrightarrow} 

\newcommand{\abs}[1]{\left\vert#1\right\vert}
\newcommand{\set}[1]{\left\{#1\right\}}

\newcommand{\brc}[1]{\left(#1\right)}
\newcommand{\brcs}[1]{\left[#1\right]}

\newcommand{\floor}[1]{\left\lfloor #1 \right\rfloor}

\newcommand{\at}[1]{\textcolor{red}{#1}}

\begin{document}

\title*{Quickest Changepoint Detection in General Multistream Stochastic Models: Recent Results, Applications and Future Challenges}
\titlerunning{Quickest Changepoint Detection  in General Stochastic Models} 
\author{Alexander G Tartakovsky and Valentin Spivak}
\institute{Alexander G Tartakovsky \at AGT StatConsult, Los Angeles, California, USA, \email{alexg.tartakovsky@gmail.com}
\and Valentin Spivak \at MIPT, Russia, \email{valyas.spivak@gmail.com}}
%
%
\maketitle

\abstract{Modern information systems generate large volumes of data with anomalies that occur at unknown points in time and have to be detected 
 quickly and reliably with low false alarm rates.  The paper develops a general theory of quickest multistream detection in non-i.i.d. stochastic models when
 a change may occur in a set of multiple data streams.  The first part of the paper focuses on the asymptotic quickest detection theory. 
 Nearly optimal pointwise detection strategies that minimize the expected detection 
 delay are proposed and analyzed when the false alarm rate is low. The general theory is illustrated in several examples. 
 In the second part, we discuss challenging applications associated with the rapid detection of new  COVID waves and 
 the appearance of near-Earth space objects.  Finally, we discuss certain open problems and future challenges.
}

%
\section{Introduction} \label{sec:intromultiple}

The problem of  changepoint detection in multiple data streams (sensors, populations, or multichannel systems) arises in numerous applications that include but are not limited to the medical sphere (detection of an 
epidemic present in only a fraction of hospitals \cite{chang, frisen-sqa09, bock,tsui-iiet12});  
environmental monitoring (detection of the presence of hazardous materials or intruders \cite{fie, mad}); 
military defense (detection of an unknown number of targets by multichannel sensor systems 
\cite{Bakutetal-book63,Tartakovsky&Brown-IEEEAES08}); near-Earth space informatics (detection of debris and satellites with telescopes
 \cite{BerenkovTarKol_EnT2020,KolessaTartakovskyetal-IEEEAES2020,TartakovskyetalIEEESP2021}); 
cyber security (detection of attacks in computer networks 
\cite{szor, Tartakovsky-Cybersecurity14, Tartakovskyetal-SM06,Tartakovskyetal-IEEESP06}); detection of malicious activity in social 
networks~\cite{Raghavanetal-AoAS2013,Raghavanetal-IEEECSS2014}, to name a few.

In many change detection applications, the pre-change  (the baseline or in-control) distribution of observed data is known, but the post-change (out-of-control) 
distribution is not completely known. As discussed in \cite[Ch 3, 6]{Tartakovsky_book2020} there are 
three conventional approaches in this case: (i)  to select a representative value of the post-change parameter and apply efficient detection rules tuned to this value such as the 
Shiryaev rule, the Shiryaev--Roberts rule or CUSUM, (ii) to select a mixing measure over the parameter space and apply mixture-type rules, 
(iii) to estimate the parameter and apply adaptive schemes. In Chapters~4 and 6 of  \cite{Tartakovsky_book2020}, a single stream case was considered. 
In this paper, we consider a more general case where the change occurs in multiple data streams and the number and location of affected 
data streams are unknown.

To be more specific, suppose there are $N$ data streams observed sequentially in time subject to a change at an unknown point in time $\nu \ge 0$, 
so that the data up to the time $\nu$ are generated by one stochastic model and after $\nu+1$ by another model. 
The change in distributions may occur at a subset of streams of a size $1 \le K \le N$, where $K$ is an assumed maximal number of streams that can be affected, which can be substantially smaller than $N$. 
A sequential detection procedure is a stopping time $T$ with respect to an observed sequence.   
A false alarm is raised when the detection is declared before the change occurs. 
One wants to detect the change with as small a delay as possible while controlling the false alarm rate. 

We consider a fairly general stochastic model assuming that the observations may be dependent and non-identically distributed (non-i.i.d.) before and 
after the change and that streams may be mutually dependent. 

In the case of  i.i.d. observations (in pre-change and post-change modes with different distributions), this problem was considered in
\cite{felsokIEEEIT2016,Mei-B2010,TartakovskyIEEECDC05,TNB_book2014,Tartakovskyetal-SM06,Xie&Siegmund-AS13}. 
Specifically, in the case of a known post-change parameter and $K=1$ (i.e., when only one stream can be affected but it is unknown which one), Tartakovsky~\cite{TartakovskyIEEECDC05} 
proposed to use a multi-chart CUSUM procedure that raises an alarm when one of the partial CUSUM statistics exceeds a threshold. This procedure is very simple, 
but it is not optimal and performs poorly when many data streams are affected. To avoid this drawback, Mei~\cite{Mei-B2010} suggested a SUM-CUSUM rule based on the sum of CUSUM 
statistics in streams and evaluated its first-order performance, which shows that this detection scheme is first-order asymptotically minimax minimizing the maximal expected delay to detection
when the average run length to false alarm approaches infinity. Fellouris and Sokolov~\cite{felsokIEEEIT2016} suggested more efficient generalized and mixture-based CUSUM rules 
that are second-order minimax.  Xie and Siegmund~\cite{Xie&Siegmund-AS13} considered a particular Gaussian model with an unknown post-change mean. 
They suggested a rule that combines mixture likelihood ratios that incorporate an assumption about the proportion of affected data streams with the generalized 
CUSUM statistics in streams and then add up the resulting local statistics. They also performed a detailed asymptotic analysis of the proposed
detection procedure in terms of the average run length to a false alarm and the expected delay as well as MC simulations. 
Chan~\cite{Chan-AS2017} studied a version of the 
mixture likelihood ratio rule for detecting a change in the mean of the normal population assuming independence of data streams and establishing its asymptotic optimality  in a minimax setting as well as 
dependence of operating characteristics on the fraction of affected streams.

In the present paper, we consider a Bayesian problem with a general prior distribution of the change point and multiple data streams 
with an unknown pattern, i.e., when the size and location of the affected streams are unknown. 
It is assumed that the observations can be dependent and non-identically distributed in data streams and even across the streams. Furthermore,
in contrast to most previous publications where asymptotically stationary models were considered, we consider substantially non-stationary models, 
even asymptotically.
We address two scenarios when the pre- and post-change distributions are completely known and also a parametric uncertainty 
when the post-change distribution is known up to an unknown parameter. We introduce mixture detection procedures that 
mix the Shiryaev--Robers-type statistic over the distributions of the unknown pattern and unknown post-change parameter (in the case of prior uncertainty).
The resulting statistics are then compared to appropriate thresholds.  

The paper is organized as follows. In Section~\ref{sec:ASQuickest}, 
we present a general theory for very general stochastic models, providing sufficient conditions under which the suggested detection procedures 
are first-order asymptotically optimal.  In Section~\ref{sec:Examples}, we provide illustrative examples. In Section~\ref{sec:MC}, 
we evaluate the performance of proposed mixture detection procedures using Monte Carlo simulations for the non-stationary Gaussian model.
In Section~\ref{sec:COVID}, theoretical results are applied to rapid detection of the COVID-19 outbreak in Australia based on monitoring the percentage of infections 
in the total population
as well as to rapid detection and extraction of faint near-Earth space objects with telescopes.
Section~\ref{sec:Remarks} concludes the paper with several remarks, including future research challenges.

\section{Asymptotic Theory of Multistream Quickest Change Detection for General Non-i.i.d.\ Models} \label{sec:ASQuickest}

In this section, we develop the quickest detection theory in the general multistream non-i.i.d. scenario. We design mixture-based change detection procedures
which are nearly optimal in the class of change detection procedures with the prespecified average (weighted) probability of false alarm when 
this probability is small, assuming that the change point is random with a given prior distribution.

\subsection{A General Multistream Model and Basic Notations}\label{ssec:Model}

We begin with a preliminary description of the scenario of interest and general notation.

Consider the multistream scenario where the observations $\Xb=(X(1), \ldots, X(N))$ are sequentially acquired in $N$ streams, i.e.,
in the $i$-th stream one observes a sequence $X(i)=\{X_n(i)\}_{n \ge1}$, where $i \in \Nc:=\{1, \ldots,N\}$. 
The observations are subject to a change at an unknown 
time $\nu\in\{0,1, 2, \dots\}$,  so that $X_1(i),\dots,X_\nu(i)$ are generated by one stochastic model and $X_{\nu+1}(i),  X_{\nu+2}(i), \dots$ 
by another model when the change occurs in the $i$-th stream.  The change in distributions happens at a subset of streams 
$\B\subseteq \{1,\dots,N\}$ with cardinality $1\le |\B| \le K \le N$, where $K$ is an assumed maximal number of streams that can be affected, 
which can be and often is substantially smaller than $N$. 
A sequential detection rule is a stopping time $T$ with respect to an observed sequence 
$\{\Xb_n\}_{n\ge 1}$, $\Xb_n=(X_n(1),\dots,X_n(N))$, i.e., $T$ is an integer-valued random variable, such that the event $\{T = n\}$
belongs to the sigma-algebra  $\Fc_{n}=\sigma(\Xb_1,\dots,\Xb_n)$ generated by observations $\Xb_1,\dots,\Xb_n$.  

The observations may have a very general stochastic structure. Specifically, if we let $\Xb^{n}(i)=(X_1(i),\dots,X_n(i))$ denote the sample of size $n$ 
in the $i$-th stream and if
$\{f_{\theta_i,n}(X_n(i)|\Xb^{n-1}(i))\}_{n\ge 1}$, $\theta_i\in\Theta_i$ is a parametric family of conditional densities of $X_n(i)$ 
given $\Xb^{n-1}(i)$, then when $\nu=\infty$ (there is no change) the parameter $\theta_i$ is equal to the known value $\theta_{i,0}$, i.e., 
$f_{\theta_i,n}(X_n(i)|\Xb^{n-1}(i))=f_{\theta_{i,0},n}(X_n(i)|\Xb^{n-1}(i))$ for all 
$n \ge 1$ and when $\nu=k<\infty$, then $\theta_i=\theta_{i,1}\neq \theta_{i,0}$, i.e., 
$f_{\theta_i,n}(X_n(i)|\Xb^{n-1}(i))=f_{\theta_{i,0},n}(X_n(i)|\Xb^{n-1}(i))$ for $n \le k$ and 
$f_{\theta_i,n}(X_n(i)|\Xb^{n-1}(i))=f_{\theta_{i,1},n}(X_{n}(i)|\Xb^{n-1}(i))$ for $n > k$. Not only the point of change $\nu$, but also the subset $\B$, 
its size $|\B|$, and the post-change parameters  $\theta_{i,1}$ are unknown. For further details with a certain change of notation see below.

 Let $\Pb_\infty$ denote the probability measure corresponding to the 
sequence of observations $\{\Xb_n\}_{n\ge 1}$  from all $N$ streams when there is never a change ($\nu=\infty$) in any of the components and,  
for $k=0,1,\dots$ and $\B \subset \Nc$, 
let $\Pb_{k,\B}$ denote the measure  corresponding to the sequence  $\{\Xb_n\}_{n\ge 1}$ when $\nu=k<\infty$ and the change occurs in a subset $\B$ of the set $\Pc$ 
(i.e., $X_{\nu+1}(i)$, $i\in \B$ is the first post-change observation). By $\Hyp_\infty: \nu=\infty$ we denote the hypothesis that the change never occurs and by $\Hyp_{k, \B}$ -- the hypothesis
that the change occurs at time $0 \le k<\infty$ is the subset of streams $\B\subset \Pc$. 
The set $\Pc$ is a  class of subsets of $\Nc$ that incorporates available prior information regarding the subset $\B$ where the change may occur.    For example,  
in applications frequently it is known that at most $K$ streams can be affected, in which case $\Pc=\Pc_K =\{\B \subset \Nc : 1 \le |\B| \le K\}$.
Hereafter $|\B|$ denotes the size of a subset $\B$ (the number of affected streams under $\Hyp_{k, \B}$) and $|\Pc|$ denotes the size of class 
$\Pc$ (the number of possible alternatives  in $\Pc$). Note that $|\Pc|$ takes maximum value when there is no prior information regarding the 
subset of affected streams, i.e., when  $\Pc=\Pc_{N}$, in which case $|\Pc|=2^{N}-1$.

The problem is to detect the change as soon as possible after it occurs regardless of the subset $\B$, i.e., we are interested in detecting the event $\cup_{\B \in \Pc} \Hyp_{k,\B}$ that the change has occurred
in some subset but not in identifying the subset of streams where it occurs.

Write $\Xb^{n}(i)=(X_1(i),\dots,X_n(i))$ for the concatenation of the first $n$ observations from the $i$-th data stream and $\Xb^{n}=(\Xb_1,\dots,\Xb_n)$ 
for the concatenation of the first $n$ observations from all $N$ data streams. Let $\{g(\Xb_n|\Xb^{n-1})\}_{n\in \Zbb_+}$ and $\{f_{\B}(\Xb_{n}|\Xb^{n-1})\}_{n\in \Zbb_+}$ 
be sequences of conditional densities of $\Xb_n$  given $\Xb^{n-1}$, which may depend on $n$, i.e., $g=g_n$ and $f_\B=f_{\B,n}$.  
We omit the subscript $n$ for the sake of brevity. For the general non-i.i.d.\  changepoint model the joint density $p(\Xb^n | \Hyp_{k,\B})$ 
under hypothesis $\Hyp_{k,\B}$ can be written as follows
\begin{equation} \label{noniidmodelpost}
p(\Xb^n | \Hyp_{k,\B})  =
\begin{cases}
 \prod_{t=1}^n g(\Xb_t|\Xb^{t-1}) \quad  & \text{for}~~ \nu=k \ge n ,
\\
\prod_{t=1}^{k}  g(\Xb_t|\Xb^{t-1}) \times \prod_{t=k+1}^{n}  f_{\B}(\Xb_t|\Xb^{t-1})  \quad & \text{for}~~ \nu=k < n,
\end{cases}
\end{equation}
where $\B\subset \Pc$. Therefore, $g(\Xb_{n}|\Xb^{n-1})$ is the pre-change conditional density and $f_{\B}(\Xb_{n}|\Xb^{n-1})$ 
is the post-change conditional density given that the change occurs in the subset $\B$.

In most practical applications, the post-change distribution is not completely known -- it depends on an unknown (generally multidimensional) parameter 
$\theta\in\Theta$,  so that the model \eqref{noniidmodelpost} may be treated only as a benchmark for a more practical case where the post-change densities 
$f_{\B}(\Xb_t|\Xb^{t-1})$ are replaced by $f_{\B,\theta}(\Xb_t|\Xb^{t-1})$, i.e.,
 \begin{align}
p(\Xb^n | H_{k,\B},\theta) & =  \prod_{t=1}^{k}  g(\Xb_t|\Xb^{t-1}) \times \prod_{t=k+1}^{n}  f_{\B,\theta}(\Xb_t|\Xb^{t-1})  \quad \text{for}~~ \nu=k < n.
\label{noniidmodelpostunknown}
\end{align}

Notice that the probabilistic models given by \eqref{noniidmodelpost} and \eqref{noniidmodelpostunknown} are very general and do not assume that
the data streams are mutually independent.

\subsection{Optimality Criterion}\label{ssec:Problem}

In the sequel, we assume that the change point $\nu$ is a random variable independent of the observations with a prior distribution 
$\pi_k=\Pb(\nu=k)$, $k=0,1,2,\dots$ with $\pi_k >0$ for $k\in\{0,1,2, \dots\}=\Zbb_+$ and that a change point may take negative values,
 but the detailed structure of the distribution $\Pb(\nu=k)$ for $k=-1,-2,\dots$ is not important. Only the total probability $\pi_{-1}=\Pb(\nu \le -1)$ 
 of the change being in effect before the observations become available matters. 

Let $\EV_{k,\B,\theta}$ and $\EV_\infty$ denote expectations under $\Pb_{k,\B,\theta}$ and $\Pb_\infty$, respectively, where $\Pb_{k,\B,\theta}$ corresponds to 
model \eqref{noniidmodelpostunknown} with an unknown parameter $\theta\in\Theta$. Define the probability measure on the Borel $\sigma$-algebra $\Bc$
in $\Rbb^\infty\times\Nbb$ as
\[
\Pb^\pi_{\B,\theta} (\Ac\times \mc{K})=\sum_{k\in \mc{K}}\,\pi_{k} \Pb_{k,\B,\theta}\left(\Ac\right), ~~\Ac\in \Bc(\Rbb^\infty), ~~ \Kc \in \Nbb.
\]  
Under measure $\Pb^\pi_{\B,\theta}$ the change point $\nu$ has distribution $\pi=\{\pi_k\}$ and the model for the observations
is of the form \eqref{noniidmodelpostunknown}, i.e., $\Xb(t)$ has conditional density $g(\Xb(t)| \Xb^{t-1})$ if $\nu \le k$ and conditional density 
$f_{\B,\theta}(\Xb(t)| \Xb^{t-1})$ if $\nu > k$ and the change occurs in the subset $\B$ with the parameter $\theta$. Let $\Eb^\pi_{\B,\theta}$ denote the expectation
under $\Pb^\pi_{\B,\theta}$.

For the prior distribution of the change point $\pi=\{\pi_k\}_{k \ge -1}$, introduce the average (weighted) probability of false alarm associated with the 
change detection procedure~$T$ 
\begin{equation} \label{PFAdefmultiple}
\PFA_\pi(T)=\Pb^\pi_{\B,\theta}( T \le \nu)= \sum_{k=0}^\infty \pi_k \Pb_\infty( T \le k) 
\end{equation}
that corresponds to the risk due to a false alarm.  Note that here we took into account 
that $\Pb_{k,\B,\theta}(T \le k) = \Pb_\infty(T\le k)$ since the event $\{T \le k\}$ depends on the observations $\Xb_1,\dots, \Xb_k$ generated by the 
pre-change probability measure $\Pb_\infty$ (recall that by our convention $\Xb_k$ is the last pre-change observation if $\nu=k$).

For $\nu=k \in \Zbb_+$, $\B\in\Pc$, and $\theta\in\Theta$ the risk associated with the detection delay is measured by the conditional expected delay to detection 
\begin{equation} \label{Riskdefmultiple}
\ADD_{k,\B,\theta}(T)=  \EV_{k, \B,\theta}\left[T-k\,|\, T> k \right] .
\end{equation}
Note that if the change occurs before the observations become available, i.e., $k \in \{-1, -2, \dots\}$, then  $\ADD_{k,\B,\theta}(T)=
\EV_{k, \B,\theta}[T] \equiv \EV_{0, \B,\theta}[T]$ since $T\ge 0$ with probability one.

Next, for the prior distribution $\pi$ and $\alpha\in (0,1)$, define the Bayesian class of changepoint detection procedures with the weighed 
probability of false alarm $\PFA_\pi(T) = \Pb^\pi_{\B,\theta}( T \le \nu)$ not greater than a prescribed number 
$\alpha$:
\begin{equation}\label{class}
\classalpi=\{T \in \Mc: \PFA_\pi(T) \le \alpha\} = \set{T\in\Mc: \sum_{k=0}^\infty \pi_k \Pb_\infty( T \le k)  \le \alpha},
\end{equation}
where $\Mc$ stands for the totality of Markov times.

In this section, we are interested  in the uniform Bayesian constrained optimization criterion
\begin{equation*}
\inf_{\{T: \PFA_\pi(T) \le \alpha\}}\,\ADD_{k,\B,\theta}(T)  \quad \text{for all} ~ k\in \Zbb_+, ~ \B\in \Pc~\text{and} ~ \theta\in\Theta.
\end{equation*}
However, this problem is intractable for arbitrary values of $\alpha\in (0, 1)$. For this reason, we will consider the following first-order 
asymptotic problem assuming that the given PFA $\alpha$ approaches zero: Find a change detection procedure $T^*$ such that it minimizes the 
expected detection delay $\ADD_{k, \B,\theta}(T)$ asymptotically to first order as $\alpha\to 0$ uniformly for all possible values of
$k\in \Zbb_+$, subsets $\B\in \Pc$, and $\theta\in\Theta$. That is, our goal it to design such detection procedure $T^*$ that, as $\alpha\to0$,
\begin{equation}\label{FOAOunifdef}
\inf_{T\in\classalpi}\ADD_{\B,\theta}(T) = \ADD_{k, \B,\theta}(T^*) (1+o(1)) ~~ 
\forall ~ \B\in \Pc,~ \theta\in\Theta, ~  \nu=k\in \Zbb_+,
\end{equation} 
where $\classalpi$ is the class of detection procedures for which the PFA does not exceed a prescribed number 
$\alpha \in (0,1)$ defined in \eqref{class} and $o(1)\to 0$ as $\alpha\to 0$.

\subsection{Multistream Changepoint Detection Procedures}\label{ssec:Proceduresunknown}

We begin by considering the most general scenario where the observations across streams are dependent.

 \subsubsection{Parametric Prior Uncertainty} \label{Uknownpar}
 
Let $\Lc_{\B,\theta}(n) = f_{\B,\theta}(\Xb_n|\Xb^{n-1})/g(\Xb_n|\Xb^{n-1})$. Note that in the general non-i.i.d.\ case the statistic 
$\Lc_{\B,\theta}(n)= \Lc_{\B,\theta}^{(k)}(n)$ depends on the change point $\nu=k$ since the post-change density  
$ f_{\B,\theta}(\Xb_n|\Xb^{n-1})= f_{\B,\theta}^{(k)}(\Xb_n|\Xb^{n-1})$ may depend on $k$.
The likelihood ratio (LR) of the hypothesis $\Hyp_{k, \B}$ that the change occurs at $\nu=k$ in the subset of streams $\B$  
against the no-change hypothesis $\Hyp_\infty$ based on the sample $\Xb^n=(\Xb_1,\dots,\Xb_n)$ is given by the product
\[
LR_{\B,\theta}(k, n) = \prod_{t=k+1}^{n}  \Lc_{\B,\theta}(t), \quad n > k
\]
and we set $LR_{\B,\theta}(k,n)=1$ for $n \le k$. 
 
 For $\B \in \Pc$ and $\theta\in\Theta$, define the generalized Shiryaev--Roberts (SR) statistic 
\begin{equation}\label{SR_stat}
\begin{split}
R_{\B,\theta}(n) =r \, LR_{\B,\theta}(0, n) + \sum_{k=0}^{n-1} LR_{\B,\theta}(k, n) , \quad n \ge 1
\end{split}
\end{equation}
with the initial condition (non-negative head-start) $R_{\B,\theta}(0)=r$, $r \ge 0$. 

Now, introduce the probability mass function (mixing measure)
\begin{equation}\label{PriorB}
\pb = \set{p_\B, \B \in \Pc}, \quad p_\B >0 ~ \forall ~ \B \in \Pc, \quad \sum_{\B \in \Pc} p_\B =1,
\end{equation}
where $p_\B$ is the prior probability of the change being in effect on the set of streams $\B$, and also the mixing probability measure
\begin{equation}\label{PriorW}
\mathbf{W}=\{W(\theta), \theta\in \Theta\}, \quad \int_{\Theta} \rm{d} W(\theta) =1.
\end{equation}

For a fixed value of $\theta$, introduce the mixture statistic
\begin{equation}\label{MSR_stat}
\begin{split}
R_{\pb,\theta}(n)& = \sum_{\B \in \Pc} p_\B R_{\B,\theta}(n)
\\
& = r  \Lambda_{\pb,\theta}(0,n) +  \sum_{k=0}^{n-1}  \Lambda_{\pb,\theta}(k,n),  ~~ n \ge 1, ~~ R_{\pb,\theta}(0)=r ,
 \end{split}
\end{equation}
where
\begin{equation}\label{MLR}
\Lambda_{\pb,\theta}(k,n) =  \sum_{\B \in \Pc} p_\B  LR_{\B,\theta}(k, n)
\end{equation}
is the mixture LR.

As discussed in \cite{Tartakovsky_book2020,TNB_book2014}, when the parameter $\theta$ is unknown there are two main conventional approaches -- 
either to estimate $\theta$ (say maximize) 
or average (mix) over  $\theta$.  Using the mixing measure  (prior distribution) $\mathbf{W}=\{W(\theta), \theta\in \Theta\}$ given in \eqref{PriorW}, 
define the double LR-mixture 
\begin{equation}\label{ALRmultiple} 
\Lambda_{\pb,W}(k,n)  =  \int_\Theta \sum_{\B \in \Pc} p_\B  LR_{\B,\theta}(k, n)  \, \mrm{d} W(\theta)
 = \int_\Theta \Lambda_{\pb,\theta}(k,n)  \, \mrm{d} W(\theta), \quad k < n
\end{equation}
and the double-mixture statistic 
\begin{equation}\label{DMSR_stat}
\begin{split}
R_{\pb,W}(n) & = \int_\Theta \sum_{\B \in \Pc} p_\B R_{\B,\theta}(n) \, \mrm{d} W(\theta)
\\
 & = r  \Lambda_{\pb,W}(0,n) +  \sum_{k=0}^{n-1}  \Lambda_{\pb,W}(k,n),  ~~ n \ge 1, ~~ R_{\pb,W}(0)= r
 \end{split}
\end{equation}
(with a non-negative head-start $r$).

The corresponding double-mixture LR-based detection procedure is given by the stopping rule which is the first time $n\ge 1$ such that the  statistic $R_{\pb,W}(n)$ exceeds the level $A>0$:
\begin{equation}\label{DMSR_def}
T_A^{\pb,W}=\inf\set{n \ge 1: R_{\pb,W}(n) \ge A}.
\end{equation}

The main result of our theory is that this changepoint detection procedure is first-order asymptotically optimal under certain very 
general conditions if threshold $A=A_\alpha$ is adequately selected, i.e., asymptotic equality
 \eqref{FOAOunifdef} holds with $T^*=T_A^{\pb,W}$.

 \subsubsection{Known Parameters  of the Post-Change Distribution}\label{Knownpar}
 
 If the value of the post-change parameter $\theta$ is known or its putative value is of special interest,
representing a nominal change, then it is reasonable to turn the double-mixture procedure
$T_A^{\pb,W}$ in single-mixture procedure $T_A^{\pb,\theta}$ by taking the degenerate weight function $W(\vartheta)$ concentrated at $\vartheta=\theta$, i.e.,  
\begin{equation}\label{MS_def}
T_A^{\pb,\theta}=\inf\set{n \ge 1: R_{\pb,\theta}(n) \ge A},
\end{equation}
and ask whether or not it has first-order asymptotic optimality properties at the point $\theta$, i.e., that under certain conditions 
asymptotic formula \eqref{FOAOunifdef} holds for $T^*=T_A^{\pb, \theta}$.
 
\subsection{Asymptotic Optimality of Mixture-Based Detection Procedures}\label{ssec:ASOPT}

\subsubsection{Basic Conditions}\label{sssec:Conditions}

While we consider a general prior and a very general stochastic model for the observations in streams and between streams, to study asymptotic optimality properties we still 
need to impose certain constraints on the prior distribution $\pi=\{\pi_k\}$ and on the general stochastic model \eqref{noniidmodelpost} that guarantee 
asymptotic stability of the detection statistics as the sample size increases. 

Regarding the prior distribution, we will assume that the following condition holds: 
\[
\lim_{n\to \infty} \frac{1}{n}  \abs{\log \Pb(\nu > n)} = \beta ~~ \text{for some} ~ \beta \ge 0.
\]
However, this condition being quite general does not cover the case where $\beta$ is positive but may go to zero. Indeed, the distributions with an exponential right tail that satisfy this condition with $\beta>0$ do not converge as $\beta\to0$ to heavy-tailed distributions with $\beta=0$. For this reason, 
any assertions for heavy-tailed distributions with $\beta=0$ do not hold if $\beta\to 0$ with an arbitrary rate; the rate must be matched with the 
PFA probability $\alpha$, $\alpha\to 0$. Hence, in what follows we consider the scenario with the prior distribution 
$\pi^\alpha=\{\pi_k^\alpha\}$ that depends on the PFA constraint $\alpha$ and modify the above condition as

\noindent $\CP_1$. For some $\beta_\alpha \ge 0$ such that $\beta_\alpha \to 0$ as $\alpha\to 0$
\[
\lim_{n\to\infty} \frac{1}{n} \abs{\log \sum_{k=n+1}^\infty \pi_k^\alpha} = \beta_\alpha.
\]

\ignore{
In addition, we will assume the following two conditions:

\noindent $\CP_2$. $\lim_{\alpha\to0} |\log \pi_k^\alpha|/|\log \alpha| = 0$ for all $k \in \Zbb_+$;

\noindent $\CP_3$. If $\beta_\alpha >0$ for all $\alpha$, then $\beta_\alpha$ goes to 0 at such rate that
\[
\lim_{\alpha\to0} \frac{\sum_{k=0}^\infty\pi_k^\alpha |\log \pi_k^\alpha|}{|\log \alpha|} =0.
\]
 }
 
For $\B \in \Pc$ and $\theta\in\Theta$, introduce the log-likelihood ratio (LLR) process between the hypotheses 
$\Hyp_{k,\B,\theta}$ ($k=0,1, \dots$) and $\Hyp_\infty$:
$$
\lambda_{\B,\theta}(k, n) = \sum_{t=k+1}^{n}\, \log \frac{f_{\B,\theta}(\Xb_t | \Xb^{t-1})}{g(\Xb_t|\Xb^{t-1})}, \quad n >k 
$$
($\lambda_{\B,\theta}(k, n)=0$ for $n \le k$). 

Let $\psi: \Rbb_+ \to \Rbb_+$ be an increasing and continuous function and by $\Psi$ denote the inverse of $\psi$, 
which is also increasing and continuous. We also assume that $\lim_{x\to\infty} \psi(x) = \infty$, and thus, $\Psi(x)$ is properly defined on the 
entire positive real line. 

In the rest of the paper, we assume that the strong law of large numbers (SLLN) holds for the LLR with the rate $\psi(n)$, that is, $\lambda_{\B,\theta}(k, k+n)/\psi(n)$
converges almost surely (a.s.) under probability $\Pb_{k,\B,\theta}$ to a finite and positive number $I_{\B,\theta}$:
\begin{equation}\label{SLLNLLR}
\frac{1}{\psi(n)} \lambda_{\B,\theta}(k, k+n) \xra[n\to\infty]{\Pb_{k,\B,\theta}-\text{a.s.}} I_{\B,\theta} ~~ \text{for all}~ k\in \Zbb_+, ~ \B\in \Pc, ~ 
\theta\in \Theta.
\end{equation}
Condition \eqref{SLLNLLR}  is the first main assumption regarding the general stochastic 
model for observed data. 

Notice that if the data in streams and across the streams are independent (but non-stationary), then
\[
I_{\B,\theta}  = \lim_{n\to\infty} \frac{1}{\psi(n)} \sum_{i \in \B} \sum_{t=k+1}^{k+n} \Eb_{k, i, \theta_i}\brcs{\frac{f_{i,\theta_i}(X_t(i)}{g_{i,t}(X_t(i))}},
\]
where $g_{i,t}(X_t(i))$ and $f_{i,\theta_i}(X_t(i))$ are pre- and post-change densities of the $t$-th observation in the $i$-th stream, respectively. 
In other words, the number $I_{\B,\theta}$ can be interpreted as the limiting local Kullback-Leibler divergence. 

If the function $\psi(n)$ is non-linear we will say that the model is {\it non-stationary (even asymptotically)}. However, in many applications the observations are 
non-identically distributed but $\psi(n)=n$, in which case we will say that the non-i.i.d.\ model is {\it asymptotically stationary}.

In Subsection~\ref{sssec:LowerBounds}, we establish the asymptotic lower bound (as $\alpha\to0$) for the minimal value of the expected detection 
delay $\ADD_{k,\B,\theta}(T)$ in class $\class_\pi(\alpha)$ whenever the almost sure convergence condition \eqref{SLLNLLR} holds. 
To obtain the lower bound it suffices to require the following right-tail condition
\[
\lim_{L\to\infty} \Pb_{k,\B,\theta} \set{\frac{1}{\psi(L)} \max_{1\le n \le L} \lambda_{\B,\theta}(k,k+n) \ge (1+\varepsilon) I_{\B,\theta}} =0 ~~ 
\text{for all}~ \varepsilon >0.
\]
This condition holds whenever the SLLN \eqref{SLLNLLR} takes place.
However, the SLLN is not sufficient to show that this lower bound is attained for the mixture detection procedures \eqref{DMSR_def} and \eqref{MS_def}. 
To this end, the SLLN should be strengthened into a complete convergence version.

\begin{definition}\label{def:Completeconv}
We say that the process $\{Y_{k, n}, n \ge 1\}$, $ k \in \Zbb_+$ converges to a random variable $Y$  uniformly completely under the measure $\Pb_{k}$ if
\[
\sum_{n=1}^\infty \sup_{k\in \Zbb_+} \Pb_{k}\set{ \abs{Y_{k,n} -Y} > \varepsilon}< \infty ~~ \text{for all}~ \varepsilon >0.
\]
\end{definition}

The proof of the upper bound presented in Subsection~\ref{sssec:AODMS} shows that if along with the SLLN \eqref{SLLNLLR} the following left-tail 
condition is satisfied 
\begin{equation}\label{LeftComplete}
\sum_{n=1}^\infty \sup_{k\in \Zbb_+} \Pb_{k,\B, \theta}\set{\frac{1}{\psi(n)}\lambda_{\B,\theta}(k, k+n) < I_{\B,\theta}  - \varepsilon} < \infty \quad \text{for all} ~ 
\varepsilon >0,
\end{equation}
where the ``information number'' $I_{\B,\theta}$ is positive and finite, then the detection procedure  \eqref{MS_def} is asymptotically optimal
for the fixed (prespecified) post-change parameter $\theta$. 

Obviously, both conditions \eqref{SLLNLLR} and \eqref{LeftComplete} are satisfied whenever there exist positive and finite
numbers $I_{\B,\theta}$ ($\B \in \Pc$, $\theta\in \Theta$) such that the normalized LLR $\lambda_{\B,\theta}(k,k+n)/\psi(n) \to  I_{\B,\theta}$
uniformly completely under $\Pb_{k, \B,\theta}$-probability. Therefore, this condition turns out to be sufficient for asymptotic optimality of 
the detection procedure  \eqref{MS_def} when the post-change parameter $\theta$ is known. In the case of the unknown post-change parameter,
the left-tail condition is similar to but more sophisticated than condition \eqref{LeftComplete}. The details will be given in Subsection~\ref{sssec:AODMSR}.

\subsubsection{Heuristic Argument} \label{sssub:Heuristics}

We begin with a heuristic argument that allows us to obtain approximations for the expected detection delay when the threshold in  
the mixture detection procedure $T_A^{\pb,W}$ is large. For the sake of simplicity, the head-start in the detection statistic
$R_{\pb,W}(n)$ is set to zero, $R_{\pb,W}(n)=r=0$. This argument also explains the reason why the condition $\CP_1$ on the prior 
distribution is imposed. 

Assume first that the change occurs at $\nu=k\le0$. It is easy to see that the logarithm of the statistic $R_{\pb,W}(n)$ can be written as
\[
\log R_{\pb,W}(n) = \lambda_{\pb,W}(0,n) + Y_n  ,
\]
 where 
 \[
 Y_n = \log \brcs{\sum_{t=1}^{n-1}\frac{\Lambda_{\pb,W}(t,n)}{\Lambda_{\pb,W}(0,n)}}.
 \]
Due to the SLLN \eqref{SLLNLLR} $\lambda_{\pb,W}(n)/\psi(n)$ converges almost surely as $n \to \infty$ under 
$\Pb_{0,\B,\theta}$ to $I_{\B,\theta}$, so we can expect that for a large $n$ 
\begin{equation}\label{LLRapprox}
 \lambda_{\pb,W}(0,n) \approx I_{\B,\theta} \, \psi(n)  + \xi_n, 
\end{equation}
where $\xi_n/\psi(n)$ converges  to $0$. Also, $Y_n$, $n\ge 1$ are ``slowly changing'' 
and converge to a random variable $Y_\infty$.  Thus, ignoring the overshoot of $\log R_{\pb,W}(T_A^{\pb,W})$ over $\log A$, we obtain
\begin{equation} \label{approxR}
\log A \approx \log R_{\pb,W}(T_A^{\pb,W}) \approx  I_{\B,\theta} \psi\brc{T_A^{\pb,W}} +\xi_{T_A^{\pb,W}} + Y_\infty .
\end{equation}
If $\psi(n)$ increases sufficiently fast, at least not slower than $n$, then the last two terms can be ignored, and hence,
\[
T_A^{\pb,W} \approx \Psi\brc{\frac{\log A}{I_{\B,\theta}}}.
\]
Taking expectation yields
\begin{equation*}
\Eb_{0,\B,\theta}[T_A^{\pb,W}] \approx \Psi\brc{\frac{\log A}{I_{\B,\theta}}}.
\end{equation*}
A similar argument leads to the following approximate formula (for a large $A$) for the expected delay $\Eb_{k,\B,\theta}[(T_A^{\pb,W}-k)^+]$ when the change occurs at $\nu=k$:
\begin{equation}\label{Expapprox}
\Eb_{k,\B,\theta}[(T_A^{\pb,W}-k)^+] \approx \Psi\brc{\frac{\log A}{I_{\B,\theta}}}, \quad k \in \Zbb_+.
\end{equation}

Next, in Lemma~\ref{Lem:PFADSR}  (see Subsection~\ref{sssec: PFA}) it is established that the probability of a false alarm of the procedure 
$T_A^{\pb,W}$ satisfies the inequality
\[
\PFA_\pi(T_A^{\pb,W}) \le \bar{\nu}_A/A,
\] 
 where $\bar{\nu}_A=\sum_{k=1}^\infty k \pi_k^A$ is the mean of the prior distribution $\pi^A$. If we assume that $\bar{\nu}_A/A \to 0$ as $A\to \infty$, 
 then this inequality along with approximate equality \eqref{Expapprox} yields the following approximation for the expected delay to detection:
\begin{equation}\label{ADDapprox}
\Eb_{k,\B,\theta}[T_A^{\pb,W}-k | T_A^{\pb,W}>k] = \frac{\Eb_{k,\theta}[(T_A^W-k)^+]}{1-\PFA_\pi(T_A^{\pb,W})} 
\approx \Psi\brc{\frac{\log A}{I_{\B,\theta}}} .
\end{equation}
In subsequent sections, this approximation is justified rigorously.

There are two key points we would like to address. The first one is that if we impose condition $\CP_1$ on the prior distribution of the change point with $\beta>0$ 
that does not depend on $\alpha$ and does not converge to $0$, then the lower bound for the expected delay to detection in class $\classalpi$
has the form
\begin{equation} \label{LBbeta}
\inf_{T\in\classalpi} \ADD_{k,\B,\theta} \ge \Psi\brc{\frac{|\log \alpha|}{I_{\B,\theta}+\beta}} (1+o(1)) \quad \text{as}~ \alpha\to 0.
\end{equation}
In this case, this lower bound is not attained by the procedure $T_A^{\pb,W}$ since if we take $A=A_\alpha\sim |\log\alpha|$, then it follows from
\eqref{ADDapprox} that
\[
\Eb_{k,\B,\theta}[T_A^{\pb,W}-k | T_A^{\pb,W}>k] \sim \Psi\brc{\frac{|\log \alpha|}{I_{\B,\theta}}} > \Psi\brc{\frac{|\log \alpha|}{I_{\B,\theta}+\beta}}.
\]
Hereafter we use a conventional notation $y_a\sim z_a$ as $a\to a_0$ if $\lim_{a\to a_0} (y_a/z_a) = 1$.
This can be expected since the detection statistic $R_{\pb,W}(n)$ is based on the uniform prior on a positive half line. However, if $\beta=\beta_\alpha\to 0$,
as we assumed in condition $\CP_1$, then the lower bound is
\[
\Psi\brc{\frac{|\log \alpha|}{I_{\B,\theta}}} (1+o(1))
\]
and it is attained by $T_A^{\pb,W}$.

Yet another key point is that to obtain the lower bound \eqref{LBbeta} with $\beta=\beta_\alpha \to 0$ but $\beta_\alpha>0$, 
as shown in the proof of Theorem~\ref{Th:LB}, we need the function $\psi(x)$ to be either linear or super-linear. Indeed, if now we define the statistic
\begin{equation*}\label{DMS_stat}
S_{\pb,W}^{\pi}(n)= \frac{1}{\Pb(\nu \ge n)} \brcs{\pi_{-1}  \Lambda_{\pb,W}(0,n) + \sum_{k=0}^{n-1} \pi_k \Lambda_{\pb,W}(k,n) } 
\end{equation*}
and the corresponding  detection procedure 
\begin{equation*}
\wtT_A^{\pb,W}=\inf\set{n \ge 1: S_{\pb,W}^\pi(n) \ge A},
\end{equation*}
then instead of \eqref{approxR} we have
\[
\log A \approx \log S_{\pb,W}(\wtT_A^{\pb,W}) \approx  \beta \wtT_A^{\pb,W} +I_{\B,\theta} \psi\brc{\wtT_A^{\pb,W}} +\xi_{\wtT_A^{\pb,W}} + Y_\infty .
\]
So if $\psi(x) <x$ is sub-linear, i.e., $\Psi(n)\gg n$ for large $n$, we expect that the prior distribution gives much more contribution than the 
observations and
\begin{equation*}
\Eb_{k,\B,\theta}[\wtT_A^{\pb,W} -k \vert \wtT_A^{\pb,W}>k ] \approx \Psi\brc{\frac{\log A}{\beta}}
\end{equation*}
as long as $\beta>0$, i.e., the prior distribution has an exponential right tail.

Despite the simplicity of the basic ideas and the approximate calculations, the rigorous argument in proving these results is rather tedious.

\subsubsection{Asymptotic Lower Bound for Expected Detection Delay}\label{sssec:LowerBounds}

For establishing asymptotic optimality of changepoint detection procedures, we first obtain, under the a.s.\ convergence condition \eqref{SLLNLLR}, 
the asymptotic lower bound for expected detection delay  
$\ADD_{k, \B,\theta}(T) = \Eb_{k, \B,\theta}\brcs{T-k| T >k}$ of any detection rule $T$ from class $\classalpi$. 
In the following subsections, we show that under certain additional conditions associated with complete convergence of LLR
these bounds are attained for the mixture procedures $T_{A}^{\pb,W}$ and $T_A^{\pb,\theta}$. 

For $\varepsilon\in(0,1)$ and $\delta>0$, define 
\begin{equation}\label{Ndef}
N_{\alpha}=N_{\alpha}(\varepsilon,\delta, \B,\theta)=  \Psi\brc{\frac{(1-\varepsilon)|\log\alpha|}{I_{\B,\theta}+\beta_\alpha+\delta}}.
\end{equation}

The following theorem specifies the asymptotic lower bound for the expected detection delay. 

\begin{theorem}\label{Th:LB}
Let the prior distribution of the change point satisfy condition $\CP_1$ and
assume that for some positive and finite numbers $I_{\B,\theta}$ ($\B\in\Pc$, $\theta\in \Theta$) the a.s.\ convergence condition 
\eqref{SLLNLLR} holds.  Suppose in addition that the function $\psi(x)$ increases not slower than $x$, i.e.,
\begin{equation}\label{Condpsi}
\psi(x) \ge x, \quad x >0 .
\end{equation}
Then for all $0<\varepsilon<1$ and $\delta >0$ 
\begin{equation}\label{Pksupclasszero}
 \lim_{\alpha\to0} \sup_{ T\in \classalpi}\Pb_{k, \B,\theta}\set{k <  T \le k+N_{\alpha}(\varepsilon,\delta, \B,\theta)} =0 ~~ \text{for all}~\B\in\Pc, \theta\in\Theta,
\end{equation}
and, as a result,  for all $\nu=k \in \Zbb_+$, $\B\in\Pc$, and $\theta\in\Theta$ 
\begin{equation}\label{LBkinclass}
\inf_{T\in\classalpi}  \ADD_{k,\B,\theta}(T) \ge \Psi\brc{\frac{|\log\alpha|}{I_{\B,\theta}}} (1+o(1)) ~~\text{as}~ \alpha\to 0.
\end{equation}
\end{theorem}

\begin{proof}
In the asymptotically stationary case where 
$\psi(n)=n$,  the methodology of the proof is analogous to that used by 
Tartakovsky \cite{TartakovskyIEEEIT2017,TartakovskyIEEEIT2019} in the proofs of the 
lower bounds in a single stream change detection problem with slightly different assumptions on the
prior distribution.  A generalization of the proof in the substantially non-stationary case has several technical 
details. We present complete proof. 

To begin, consider an arbitrary stopping time $T\in\classalpi$ and note that by Markov's inequality
\begin{align*}
 \ADD_{k,\B,\theta}(T)  & \ge  \Eb_{k, \B,\theta}\brcs{(T-k)^+} \ge N_\alpha \, \Pb_{k,\B,\theta}\set{(T-k)^+ > N_\alpha} 
 \\
 & =  N_\alpha \, \Pb_{k,\B,\theta}\set{T > k+ N_\alpha}  .
\end{align*}
If assertion \eqref{Pksupclasszero} holds, then
\begin{equation}\label{Pkinf0}
\lim_{\alpha\to0} \inf_{T\in\classalpi} \Pb_{k,\B,\theta}\set{T > k+N_\alpha} = 1.
\end{equation}
Indeed,
\begin{align*}
\Pb_{k,\B,\theta}\set{T > k+ N_\alpha} = \Pb_{\infty}\set{T > k} - \Pb_{k,\B,\theta}\set{k< T  \le  k+ N_\alpha},
\end{align*}
where we used the identity $\Pb_{k, \B,\theta}( T>k)=\Pb_\infty (T>k)$. Next, since for any change detection procedure $T\in\classalpi$,
\[
\alpha \ge \sum_{i=k}^\infty \pi_i^\alpha \Pb_\infty(T \le i) \ge \Pb_\infty(T \le k) \sum_{i=k}^\infty \pi_i^\alpha ,
\]
it follows that
\begin{equation}\label{Psup}
\inf_{T\in\classalpi} \Pb_\infty(T > k) \ge 1- \alpha/\Pi_{k-1}^\alpha, \quad k \in \Zbb_+,
\end{equation}
where $\Pi_{\ell}^\alpha = \Pb(\nu> \ell)=\sum_{k=\ell+1}^\infty \pi_k^\alpha$. Hence, we obtain
\[
\inf_{T\in \classalpi} \Pb_{k,\B,\theta}\set{T > k+ N_\alpha} \ge 1- \alpha/ \Pi_{k-1}^\alpha - \sup_{ T\in \classalpi}\Pb_{k, \B,\theta}(k <  T \le k+N_{\alpha}).
\]
This inequality yields \eqref{Pkinf0}, and we obtain the asymptotic inequality
\[
 \ADD_{k,\B,\theta}(T) \ge N_\alpha (1+o(1)) ~~ \text{as}~ \alpha\to 0,
\]
 which holds for arbitrary values of $\varepsilon \in (0,1)$ and $\delta> 0$. By our assumption, the function 
 $N_\alpha=N_{\alpha}(\varepsilon,\delta, \B,\theta)$ is continuous, so we may take a limit $\delta, \varepsilon \to 0$, which implies
 inequality \eqref{LBkinclass}. Consequently, to show the validity of asymptotic  inequality \eqref{LBkinclass}, it suffices to prove the equality 
 \eqref{Pksupclasszero}.  
 
 Let 
 \[
\lambda^*_{\B, \theta}(t) = \log \frac{f_{\B,\theta}(\Xb_t | \Xb^{t-1})}{g(\Xb_t|\Xb^{t-1})}
 \]
 and notice that for $n>k$
 \[
 \frac{\drm \Pb_\infty^{(n)}}{\drm \Pb_{k,\B,\theta}^{(n)}} = \frac{1}{LR_{\B,\theta}(k, n)} = \exp\set{-\sum_{t=k+1}^n \lambda^*_{\B,\theta}(t) },
 \]
where hereafter $\Pb^{(n)}$ denotes a restriction of the measure $\Pb$ to the sigma-algebra $\Fc_n$. Therefore, changing the measure 
$\Pb_\infty \to \Pb_{k, \B,\theta}$ and using Wald's likelihood ratio identity, we obtain for any $C >0$:
\begin{align*}
& \Pb_{\infty}\set{k < T \le k + N_\alpha} =   \EV_{k,\B,\theta}\brcs{{\ind{0 < T -k \le N_\alpha}  \frac{\drm \Pb_\infty^{(T)}}{\drm \Pb_{k,\B,\theta}^{(T)}}  }}
\\
&= \EV_{k,\B,\theta}\brcs{{\ind{0 < T -k \le N_\alpha} e^{- \sum_{t=k+1}^T \lambda^*_{\B,\theta}(t) } }}
\\
& \ge \EV_{k,\B,\theta}\brcs{{\ind{0 < T -k \le N_\alpha, \sum_{t=k+1}^T \lambda^*_{\B,\theta}(t) < C} e^{- \sum_{t=k+1}^T \lambda^*_{\B,\theta}(t)}}}
\\
&\ge e^{-C} \Pb_{k,\B,\theta}\set{0 < T -k \le N_\alpha, \max_{0 < n-k \le  N_\alpha }\sum_{t=k+1}^n \lambda^*_{\B,\theta}(t) < C}
\\
&\ge e^{-C} \brcs{\Pb_{k,\B,\theta}\set{0 < T -k \le N_\alpha}  - 
\Pb_{k,\B,\theta}\set{\max_{0 \le n <  N_\alpha }\sum_{t=k+1}^{k+n+1} \lambda^*_{\B,\theta}(t) \ge C}} ,
\end{align*} 
where the last inequality follows from the fact that $\Pr({\cal A} \cap {\cal B}) = \Pr({\cal A}) - \Pr({\cal B}^c)$ for any events ${\cal A}$ and ${\cal B}$, where 
${\cal B}^c$ is the complement event of ${\cal B}$. Setting 
\[
C=\psi(N_\alpha) I_{\B,\theta}(1+\varepsilon) = (1+\varepsilon)K_\alpha, \quad 
K_\alpha= \frac{(1-\varepsilon) |\log\alpha| I_{\B,\theta}}{I_{\B,\theta}+\beta_\alpha+\delta} 
\]
yields
\begin{equation}\label{Probnuh}
\Pb_{k,\B,\theta}\set{k < T \le k + N_\alpha} \le \kappa_{k,\alpha}(T) + \sup_{k \ge 0} \gamma_{k,\alpha},
\end{equation}
where
\[
\kappa_{k,\alpha}(T) =\kappa_{k,\alpha}(\varepsilon, \delta, \theta, \B, T)  = e^{(1+\varepsilon) K_\alpha} \Pb_{\infty}\set{0 < T -k \le N_\alpha}
\]
and
\[
\gamma_{k,\alpha}= \gamma_{k,\alpha}(\varepsilon,\B,\theta) = 
\Pb_{k,\B,\theta}\set{\max_{0 \le n <  N_\alpha }\sum_{t=k+1}^{k+n+1} \lambda^*_{\B,\theta}(t) \ge (1+\varepsilon) I_{\B,\theta}\psi(N_\alpha)}.
\]
 By the a.s. convergence condition \eqref{SLLNLLR}, 
\begin{equation}\label{subbetato0}
\lim_{\alpha\to 0} \sup_{k \ge 0} \gamma_{k,\alpha} = 0.
\end{equation}

To evaluate $\kappa_{k,\alpha}(T)$  it suffices to note that, by condition $\CP_1$, for all sufficiently small $\alpha$, there exists a small $\delta$ such that
\[
\frac{|\log \Pi^\alpha_{k-1+N_\alpha}|}{ k-1+N_{\alpha}} \le \beta_\alpha + \delta,
\]
and to use inequality  \eqref{Psup}, which yields that for all sufficiently small $\alpha$
\begin{align*} 
\kappa_{k,\alpha}( T)  & \le e^{(1+\varepsilon) K_\alpha}  \Pb_\infty( T \le k+N_{\alpha}) 
\le \alpha \, e^{(1+\varepsilon)K_\alpha}/\Pi^\alpha_{k-1+N_\alpha} 
\\
& \le \exp\set{(1+\varepsilon)K_\alpha -|\log\alpha| +( k-1+N_{\alpha}) \frac{|\log\Pi^\alpha_{k-1+N_\alpha}|}{ k-1+N_{\alpha}}} 
\\
& \le  \exp\set{(1+\varepsilon)K_\alpha-|\log\alpha| +( k-1+N_{\alpha}) (\beta_\alpha+\delta)} .
\end{align*}
Since by condition \eqref{Condpsi}, 
\[ 
N_\alpha \le K_\alpha/I_{\B,\theta}=\frac{(1-\varepsilon)|\log\alpha|}{I_{\B,\theta}+\beta_\alpha+\delta}
\]
it follows that
\begin{align*} 
\kappa_{k,\alpha}( T)  & \le 
\exp\set{-\frac{I_{\B,\theta}\varepsilon^2 + (\beta_\alpha+\delta) \varepsilon}{I_{\B,\theta}+\beta_\alpha + \delta}|\log\alpha| + 
(\beta_\alpha +\delta) k}
\\
& \le \exp\set{- \varepsilon^2 |\log\alpha| + (\beta_\alpha+\delta)k} := \overline{\kappa}_{\alpha,k}(\varepsilon,\delta).
\end{align*}
for all $\varepsilon \in (0,1)$ and every small $\delta$. Hence,  for all $\varepsilon \in (0,1)$ and sufficiently small $\delta$,
\begin{equation} \label{IneqU}
\kappa_{k,\alpha}(T)  \le  \overline{\kappa}_{k,\alpha}(\varepsilon,\delta) ,
\end{equation}
where $\overline{\kappa}_{k,\alpha}(\varepsilon,\delta)$ does not depend on the stopping time $T$ and goes to $0$ as $\alpha\to0$ 
for any fixed $k\in \Zbb_+$, which implies that 
\begin{equation*}
\sup_{T\in\classalpi} \kappa_{k,\alpha}( T) \to 0 \quad \text{as}~ \alpha\to0 ~~ \text{for every fixed}~ k \in\Zbb_+.
\end{equation*}
 The proof of the lower bound \eqref{LBkinclass} is complete.   \qed 
\end{proof}

\subsubsection{Probabilities of False Alarm of Mixture Change Detection Procedures}\label{sssec: PFA}

An important question is how to select threshold $A$ in change detection procedures $T_A^{\pb,\theta}$ and $T_A^{\pb,W}$ defined in 
\eqref{MS_def} and \eqref{DMSR_def}, respectively, to imbed them into class $\classalpi$.

The following lemma provides the upper bounds for the PFA of these procedures.

\begin{lemma}\label{Lem:PFADSR}  
For all $A >0$ and any prior distribution $\pi=\{\pi_k\}$ of the change point $\nu$ with finite mean $\bar\nu=\sum_{k=1}^\infty k \pi_k$, the weighted probabilities of false alarms of 
the detection procedures $T_A^{\pb,\theta}$  and  $T_{A}^{\pb,W}$ satisfy the inequalities 
\begin{align}
\PFA_\pi(T_A^{\pb,\theta}) & \le \frac{r b + \bar{\nu}} {A} ,
\label{PFAMSRineq}
\\
\PFA_\pi(T_A^{\pb,W})  & \le \frac{r b + \bar{\nu}} {A} ,
\label{PFADMSRineq}
\end{align}
where $b=\sum_{k=1}^\infty \pi_k$. Hence,  if $A= A_\alpha =   (r b+\bar\nu)/\alpha$,
then $\PFA_\pi(T_{A_\alpha}^{\pb,\theta}) \le \alpha$ and  $\PFA_\pi(T_{A_\alpha}^{\pb,W}) \le \alpha$, i.e., 
$T_{A_\alpha}^{\pb,\theta}\in \classalpi$ and $T_{A_\alpha}^{\pb,W} \in \classalpi$.  
\end{lemma}

\begin{proof}
We have  
\begin{align*}
\Eb_\infty[R_{\pb,\theta}(n)| \Fc_{n-1}]   =\sum_{\B\in\Pc} p_\B +\sum_{\B\in\Pc} p_\B R_{\B,\theta}(n-1)
= 1+ R_{\pb,\theta}(n-1), 
\end{align*}
and hence,  the statistic $\{R_{\pb,\theta}(n)\}_{n\in\Zbb_+}$ is a non-negative $(\Pb_\infty,\Fc_n)$-submartingale  with mean  
\[
\Eb_\infty [R_{\pb,\theta}(n)]=r+n, \quad n \in \Zbb_+.
\]
By Doob's maximal submartingale inequality, for any $ k \ge 1$,
\begin{equation}\label{PFADoobineq}
 \Pb_\infty(T_A^{\pb,\theta} \le k) = \Pb_\infty\set{\max_{0 \le n \le k} R_{\pb,\theta}(n) \ge A} \le \frac{\Eb_\infty[R_{\pb,\theta}(k)]}{A} 
 =\frac{r+k}{A} ,
 \end{equation}
which implies 
\[
\sum_{k=1}^\infty \pi_k \Pb_\infty(T_A^{\pb,\theta} \le k) \le \frac{r \sum_{k=1}^\infty \pi_k+\sum_{k=1}^\infty k \pi_k }{A},
\]
and inequality \eqref{PFAMSRineq} follows.

To prove inequality \eqref{PFADMSRineq} it suffices to note that the double mixture statistic  $\{R_{\pb,W}(n)\}_{n\ge 1}$ is also
a $(\Pb_\infty,\Fc_n)$-submartingale  with mean $\Eb_\infty [R_{\pb,W}(n)=r+n$ since
\begin{align*}
\Eb_\infty[R_{\pb,W}(n)| \Fc_{n-1}] & =   \Eb_\infty \brcs{\int_\Theta R_{\pb,\theta}(n)  \mrm{d} W(\theta) | \Fc_{n-1}} 
= 
\int_\Theta  \Eb_\infty  \brcs{R_{\pb,\theta}(n)  | \Fc_{n-1}}  \mrm{d} W(\theta) 
\\
&= 1+  \int_\Theta R_{\pb,\theta}(n-1)   \mrm{d} W(\theta)
= 1+ R_{\pb,W}(n-1).
\end{align*}
Applying Doob's submartingale inequality yields
\begin{equation}\label{PFADoobineq2}
 \Pb_\infty(T_A^{\pb,W} \le k) \le (r+k)/A, \quad k=1,2,\dots ,
 \end{equation}
which implies  \eqref{PFADMSRineq} and the proof is complete.\qed
\end{proof}

\subsubsection{Asymptotic Optimality of Mixture Detection Procedure $T_A^{\pb,\theta}$ for Known Post-change Parameter }\label{sssec:AODMS} 

We now proceed with establishing asymptotic optimality properties of the mixture detection procedures $T_A^{\pb,\theta}$ in 
class $\classalpi$ as $\alpha\to0$ assuming that the post-change parameter $\theta$ is prespecified (known or selected).

Throughout this subsection, we assume that there exist positive and finite numbers $I_{\B,\theta}$ ($\B \in \Pc$, $\theta\in \Theta$) 
such that the normalized LLR $\lambda_{\B,\theta}(k,k+n)/\psi(n)$ converges as $n\to\infty$ to  $I_{\B,\theta}$
uniformly completely under probability measure $\Pb_{k, \B,\theta}$, i.e.,
\begin{equation}\label{LLRunifcomplete}
\sum_{n=1}^\infty \sup_{k\in\Zbb_+}\Pb_{k, \B,\theta}\set{\abs{\frac{\lambda_{\B,\theta}(k,k+n)}{\psi(n)} - I_{\B,\theta}}> \varepsilon} < \infty 
\quad \text{for all}~ \varepsilon >0 .
\end{equation}

\paragraph{\bf \ref{sssec:AODMS}(a) Asymptotic Operating Characteristics of Detection Procedure $T_A^{\pb,\theta}$ for Large Threshold Values}\label{sssec:OperChar}

The following theorem provides asymptotic operating characteristics of the mixture detection procedure $T_A^{\pb,\theta}$ for large values of 
threshold $A$. Since no constraints are imposed on the false alarm rate (FAR) -- neither on the PFA nor any other FAR measure -- 
this result is universal and can be used in a variety of optimization problems.  We need to impose some constraints on the behavior of the head-start $r_A$, 
which may depend on $A$ and approach infinity as $A\to\infty$ in such a way that
\begin{equation}\label{rA}
\lim_{A\to \infty} (r_A \, A^{-\varepsilon}) = 0 \quad \text{for any}~ \varepsilon >0.
\end{equation} 
We also need to assume that the function $\psi(x)$ increases not too slowly, at least faster than the logarithmic function. Specifically, 
we assume the following condition on the inverse function $\Psi(x)$
\begin{equation}\label{CondPsi}
\lim_{x\to\infty} \frac{\log \Psi(x)}{x} =0.
\end{equation}

\begin{theorem}\label{Th:AOCMSR} 
Suppose that conditions \eqref{rA} and \eqref{CondPsi} hold and there exist positive and finite numbers 
$I_{\B,\theta}$, $\B\in\Pc$, $\theta\in\Theta$, such that 
the uniform complete convergence condition \eqref{LLRunifcomplete} is satisfied. Then the following asymptotic as $A\to\infty$ approximation holds
\begin{equation} \label{ADDMSR}
 \ADD_{k,\B,\theta}(T_A^{\pb, \theta}) = \Psi\brc{\frac{\log A}{I_{\B,\theta}}} (1+o(1))\quad  \text{for all}~ k \in \Zbb_+, ~ \B\in\Pc.
\end{equation}
\end{theorem}

\begin{proof}
For $\varepsilon\in(0,1)$, let $N_A=N_{A}(\varepsilon,\B,\theta) = \Psi\brc{(1-\varepsilon) (\log A)/I_{\B,\theta}}$. 
Inequality \eqref{PFADoobineq} implies that for all $k \in \Zbb_+$
\[
\Pb_{k, \B,\theta}(T_A^{\pb,\theta}>k)=\Pb_\infty (T_A^{\pb,\theta}>k) \ge 1- \frac{k+ r_A} {A} .
\]
Hence, using Markov's inequality, we obtain
\begin{align} \label{LBSRA}
\ADD_{k,\B,\theta}(T_A^{\pb,\theta}) &\ge \Eb_{k,\B,\theta} \brcs{(T_A^{\pb,\theta}-k)^+} \nonumber
\\
& \ge N_{A} \, \Pb_{k, \B,\theta}(T_A^{\pb,\theta}  -k > N_{A})  \nonumber
\\
&\ge N_{A}\brcs{\Pb_{k,\B,\theta}(T_A^{\pb,\theta}>k) - \Pb_{k, \B,\theta}(k <  T_A^{\pb,\theta} < k+N_{A})}  \nonumber
\\
&\ge  N_{A}\brcs{1- \frac{r_A  +k } {A} - \Pb_{k, \B,\theta}(k <  T_A^{\pb,\theta} < k+N_{A})}.
\end{align}

Similarly to \eqref{Probnuh}
\begin{equation}\label{ProbkNA}
\Pb_{k,\B,\theta}\set{k < T_A^{\pb,\theta} \le k + N_A} \le \kappa_{k,A}(T_A^{\pb,\theta}) + \gamma_{k,A},
\end{equation}
where
\[
\kappa_{k,A}(T_A^{\pb,\theta}) = e^{(1-\varepsilon^2) \log A} \Pb_{\infty}\set{0 < T_A^{\pb,\theta} -k \le N_A}
\]
and
\[
\gamma_{k,A}=  
\Pb_{k,\B,\theta}\set{\max_{0 \le n <  N_A } \lambda_{\B,\theta}(k, k+n) \ge (1+\varepsilon) I_{\B,\theta}\psi(N_A)}.
\]
 Using inequality \eqref{PFADoobineq}, we obtain
\[
\Pb_\infty\brc{0 < T_A^{\pb,\theta} - k <N_{A}} \le \Pb_\infty\brc{T_A^{\pb,\theta}  < k+ N_{A}} \le (k+r_A + N_{A})/A,
\]
and consequently,
\begin{equation}\label{UpperU}
 \kappa_{k,A}(T_A^{\pb,\theta}) \le \frac{k+ r_A+\Psi((1-\varepsilon) I_{\B,\theta}^{-1} \log A)}{A^{\varepsilon^2}} .
\end{equation}
Since, by conditions \eqref{rA} and \eqref{CondPsi},  for any $\varepsilon > 0$
\[
\frac{r_A+\Psi((1-\varepsilon) I_{\B,\theta}^{-1} \log A)}{A^{\varepsilon^2}} \to 0 \quad \text{as}~ A \to \infty.
\]
it follows from \eqref{UpperU}
that  $ \kappa_{k,A}(T_A^{\pb,\theta})\to 0$ as $A\to\infty$ for any $k\in \Zbb_+$. Also, $\gamma_{k,A} \to 0$ as $A\to\infty$
by condition \eqref{LLRunifcomplete}. Therefore,
\[
\Pb_{k, \B,\theta}\brc{0 < T_A^{\pb,\theta} -k < N_{A}}\to0 \quad \text{as}~ A \to \infty
\]
for any fixed $k$. It follows from \eqref{LBSRA}  that  as $A\to\infty$
\begin{equation*}
\ADD_{k,\B,\theta}(T_A^{\pb,\theta}) \ge \Psi\brc{\frac{(1-\varepsilon) \log A}{I_{\B,\theta}}} (1+o(1)),
\end{equation*}
where  $\varepsilon$ can be arbitrarily small. This yields the asymptotic lower bound (for any fixed $k\in\Zbb_+$, $\B\in\Pc$)
\begin{equation}\label{LBSRAasympt}
\liminf_{A\to\infty}\frac{\ADD_{k,\B,\theta}(T_A^{\pb,\theta})}{\log A} \ge \frac{1}{I_{\B,\theta}} .
\end{equation}

To prove \eqref{ADDMSR} it suffices to show that this bound is attained by $T_A^{\pb,\theta}$, i.e., 
\begin{equation}\label{UBMSR}
 \limsup_{A\to\infty} \frac{\ADD_{k, \B,\theta}(T_A^{\pb,\theta})}{\log A}  \le \frac{1}{I_{\B,\theta}} .
\end{equation}

For $0<\varepsilon < I_{\B,\theta}$, define 
\begin{equation}\label{MA}
M_{A}= M_{A}(\varepsilon,\B,\theta)=1+\floor{\Psi\brc{\frac{\log A}{I_{\B,\theta}-\varepsilon}}} .
\end{equation}
 Recall the definitions of the mixture SR statistic $R_{\pb,\theta}(n)$ and mixture LR $\Lambda_{\pb, \theta}(n)$ given in \eqref{MSR_stat} and \eqref{MLR},
 respectively. Obviously, for any $n \ge 1$ and $\B\in \Pc$,
\[
\log R_{\pb,\theta}(k+n) \ge  \log \Lambda_{\pb,\theta}(k,k+n)  \ge  \lambda_{\B,\vartheta}(k, k+n) +\log p_\B ,
\] 
and we have
\begin{align*}
&\Pb_{k, \B,\theta}\brc{T_A^{\pb,\theta}-k >n} \le \Pb_{k, \B,\theta}\set{\frac{1}{\psi(n)} \log R_{\pb,\theta}(k+n) < \frac{1}{\psi(n)} \log A}
\\
& \le \Pb_{k, \B,\theta}\set{\frac{1}{\psi(n)} \lambda_{\B,\theta}(k, k+n)  <  \frac{1}{\psi(n)} \brc{\log A  + |\log p_\B|}} ,
\end{align*}
where for any $n\ge M_{A}$ the last probability does not exceed the probability
\[
\Pb_{k, \B,\theta}\set{\frac{1}{\psi(n)} \lambda_{\B,\theta}(k, k+n) <  I_{\B,\theta} - \varepsilon + |\log p_\B|/\psi(M_A)}.
\]
Therefore,  for $I_{\B,\theta}> \varepsilon >0$, all $n \ge M_A$ and all sufficiently large  $A$ such that $|\log p_\B|/\psi(M_A) < \varepsilon/2$ , we have
\begin{align}
 \Pb_{k, \B,\theta}\brc{T_A^{\pb,\theta}-k >n} & \le 
 \Pb_{k, \B,\theta}\set{\frac{1}{\psi(n)} \lambda_{\B,\theta}(k, k+n) <  I_{\B,\theta} - \varepsilon/2}.
 \label{ProbkTAW}
\end{align}

Now, by Lemma~A1 in Tartakovsky~\cite[page 239]{Tartakovsky_book2020}, for any $\ell \ge 1$, $k\in\Zbb_+$, and $\B\in\Pc$, 
\begin{align}\label{Ektildetauineq}
& \Eb_{k, \B,\theta}\brcs{(T_A^{\pb,\theta}-k)^+}^\ell   \le 
 M_{A}^{\ell} + \ell \, 2^{\ell-1} \sum_{n= M_{A}}^{\infty}  n^{\ell-1}   \Pb_{k, \B,\theta}\brc{T_A^{\pb,\theta} >  n},
 \end{align}
 which along with inequality \eqref{ProbkTAW} yields
\begin{align}
\Eb_{k,\B,\theta}\brcs{\brc{T_A^{\pb,\theta}-k}^+} & \le   1+ \floor{\Psi\brc{\frac{\log A}{I_{\B,\theta}-\varepsilon}}} \nonumber
\\
&+ \sum_{n=M_A}^\infty  \Pb_{k, \B,\theta}\set{\frac{1}{\psi(n)} \lambda_{\B,\theta}(k, k+n) <  I_{\B,\theta} - \varepsilon/2}.
 \label{EkTAupper}
 \end{align}

Write 
\[
\Upsilon_A(\varepsilon,\B, \theta) =  \sum_{n=M_A}^\infty \sup_{k \in \Zbb_+}\Pb_{k, \B,\theta}\set{\frac{1}{\psi(n)} \lambda_{\B,\theta}(k, k+n) 
<  I_{\B,\theta} - \varepsilon}.
\]
Using \eqref{EkTAupper} and  the inequality $\Pb_\infty(T_A^{\pb,\theta} > k) > 1- (r_A +k)/A$ (see \eqref{PFADoobineq}), we obtain
\begin{align}
  \ADD_{k, \B,\theta}(T_A^{\pb,\theta})& =\frac{\Eb_{k, \B,\theta}\brcs{\brc{T_A^{\pb,\theta}-k}^+}}{\Pb_\infty(T_A^{\pb,\theta} > k)}  \nonumber
  \\
  &\le   \frac{1+\floor{\Psi\brc{\frac{\log A}{I_{\B,\theta}-\varepsilon}}} + \Upsilon_A(\varepsilon/2,\B, \theta)}{1- (r_A +k)/A}. 
  \label{ADDupperMSR}
 \end{align}
Since due to condition \eqref{rA} $r_A/A\to 0$ and, by complete convergence condition \eqref{LLRunifcomplete},  
$\lim_{A\to\infty} \Upsilon_A(\varepsilon/2,\B, \theta) = 0$ for all $\varepsilon>0$ and $\B\in\Pc$, 
 inequality \eqref{ADDupperMSR} implies the asymptotic inequality
\[
 \ADD_{k, \B,\theta}(T_A^{\pb,\theta})\le \brc{\frac{\log A}{I_{\B,\theta} - \varepsilon}} (1+o(1)) \quad \text{as}~A \to \infty.
\]
Since $\varepsilon$ can be arbitrarily small the asymptotic upper bound \eqref{UBMSR} follows
and the proof of the asymptotic approximation  \eqref{ADDMSR} is complete. \qed
\end{proof}

\paragraph{\bf \ref{sssec:AODMS}(b) Asymptotic Optimality}\label{sssec:AOTWA}

Since we assume that the prior distribution may depend on the prescribed PFA $\alpha$ the mean 
$\bar\nu_\alpha = \sum_{k=0}^\infty k \, \pi_k^\alpha$ depends on $\alpha$. We also suppose that the head-start  
$r=r_\alpha$ may depend on $\alpha$ and may go to infinity as $\alpha\to0$. We further assume that $r_\alpha$ and $\bar\nu_\alpha$ approach
infinity with such a rate that
\begin{equation}\label{PriorSRalpha}
\lim_{\alpha\to 0} \frac{{\log (r_\alpha+\bar\nu_\alpha)}}{|\log \alpha|} = 0.
\end{equation} 

The following theorem establishes the first-order asymptotic (as $\alpha\to0$) optimality of the mixture detection procedure 
$T_A^{\pb,\theta}$ for the fixed value of $\theta$ in the general non-i.i.d.\ case.  In this theorem, we assume that $\psi(x)$ satisfies condition 
\eqref{Condpsi} in contrast to Theorem~\ref{Th:AOCMSR} where this function can be practically arbitrary.

\begin{theorem}\label{Th:FOAOMSR} 
Let the prior distribution of the change point satisfy condition $\CP_1$ and let the mean of the prior distribution $\nu_\alpha$ and 
the head-start $r_\alpha$ of the statistic $R_{\pb,\theta}(n)$ go to infinity with such a rate that condition \eqref{PriorSRalpha} hold.  Assume that for some 
$0<I_{\B,\theta}<\infty$ ($\B\in\Pc$) the uniform complete convergence condition \eqref{LLRunifcomplete} holds. If  threshold
$A=A_\alpha$ is so selected that $\PFA_\pi(T_{A_\alpha}^{\pb,\theta}) \le \alpha$ and $\log A_\alpha\sim |\log\alpha|$ as $\alpha\to0$, 
in particular as $A_\alpha=(r_\alpha +\bar{\nu}_\alpha)/\alpha$, then for all $k\in \Zbb_+$ and $\B\in\Pc$ as $\alpha\to0$
\begin{align}
\ADD_{k,\B,\theta}(T_{A_\alpha}^{\pb,\theta}) &= \Psi\brc{\frac{|\log\alpha|}{I_{\B,\theta}}} (1+o(1));
\label{FOAOMSR}
\\
 \inf_{T \in \classalpi}\ADD_{k,\B,\theta}(T) & =  \ADD_{k,\B,\theta}(T_{A_\alpha}^{\pb,\theta})(1+o(1)).
 \label{FOAOMSR2}
 \end{align}
That is, the detection procedure $T_{A_\alpha}^{\pb,\theta}$ is first-order asymptotically optimal as 
$\alpha\to0$ in class $\classalpi$,  minimizing average detection delay  $\ADD_{k,\B,\theta}(T)$ uniformly for all $k\in \Zbb_+$ and $\B\in\Pc$.
\end{theorem}

\begin{proof}
If threshold $A_\alpha$ is so selected that  $\log A_\alpha\sim |\log\alpha|$ as $\alpha\to0$, then asymptotic approximation \eqref{FOAOMSR}
follows from Theorem~\ref{Th:AOCMSR}. 
This asymptotic approximation coincides with the asymptotic lower bound \eqref{LBkinclass} in Theorem~\ref{Th:LB} if condition \eqref{Condpsi} 
on $\psi(x)$ holds. Hence, the lower bound is attained 
by $T^{\pb,\theta}_{A_\alpha}$, proving \eqref{FOAOMSR2} and the assertion of the theorem.

In particular, if threshold is selected as  $A_\alpha=(r_\alpha +\bar{\nu}_\alpha)/\alpha$, then by Lemma~\ref{Lem:PFADSR} 
$\PFA_\pi(T_{A_\alpha}^{\pb,\theta}) \le \alpha$ and $\log A_\alpha\sim -\log \alpha$ due to condition \eqref{PriorSRalpha}. \qed
\end{proof}

\begin{remark}\label{Rem:MSR}
If the prior distribution $\pi$ of the change point does not depend on $\alpha$ and condition $\CP_1$ is satisfied with $\beta \ge 0$, then 
the assertion of Theorem~\ref{Th:FOAOMSR} holds if, and only if, $\beta =0$, i.e., for heavy-tailed prior distributions, but not for priors with
exponential right tail with $\beta>0$. This follows from the fact that for $\beta>0$ the lower bound has the form
\[
\inf_{T\in \classalpi} \ADD_{k,\B,\theta}(T) \ge \Psi\brc{\frac{|\log \alpha|}{I_{\B,\theta} + \beta}} (1+o(1) \quad \text{as}~ \alpha \to 0 .
\]
\end{remark}

\subsubsection{Asymptotic Optimality of the Double-Mixture Procedure $T_A^{\pb,W}$ for Unknown Post-change Parameter}\label{sssec:AODMSR}

Consider now the case where the post-change parameter $\theta\in\Theta$ is unknown. The goal is to show that the double-mixture detection 
procedure $T_A^{\pb,W}$ defined in \eqref{DMSR_def} is asymptotically optimal to first order. 

Recall that in the case of the known parameter $\theta$, the sufficient condition for asymptotic optimality (imposed on the data model) is
the uniform complete version of the SLLN \eqref{LLRunifcomplete}. The proofs of Theorem~\ref{Th:LB} and Theorem~\ref{Th:AOCMSR} 
show that to establish asymptotic optimality of procedure $T_A^{\pb,\theta}$ it suffices to require the following two (right-tail and left-tail) conditions: 
for all $\varepsilon>0$,  $k\in\Zbb_+$ and $\B\in\Pc$
\begin{equation}\label{RTcond}
\lim_{L\to\infty} \Pb_{k,\B,\theta} \set{\frac{1}{\psi(L)} \max_{0 \le n < L} \lambda_{\B,\theta}(k, k+n) \ge (1+\varepsilon) I _{\B,\theta}} =0
\end{equation}
and
\begin{equation} \label{LTcond}
\sum_{n=1}^\infty  \Pb_{k, \B,\theta}\set{\frac{1}{\psi(n)} \lambda_{\B,\theta}(k, k+n) <  I_{\B,\theta} - \varepsilon} < \infty.
\end{equation}
Note that the SLLN \eqref{SLLNLLR} guarantees right-tail condition \eqref{RTcond} and complete version \eqref{LLRunifcomplete} 
guarantees both conditions \eqref{RTcond} and \eqref{LTcond}.

If the post-change parameter $\theta$ is unknown to obtain the upper bound for the expected detection delay the left-trail condition \eqref{LTcond}
has to be modified as follows. For $\delta>0$, define $\Gamma_{\delta,\theta}=\{\vartheta\in\Theta\,:\,\vert \vartheta-\theta\vert<\delta\}$ and
assume that there exist positive and finite numbers $I_{\B,\theta}$ ($\B \in \Pc$, $\theta\in \Theta$) such that for any $\varepsilon>0$ and for all
$\B\in\Pc$ and $\theta\in\Theta$
\begin{equation}\label{compLeft}
\lim_{\delta\to0}
\sum_{n=1}^\infty \sup_{k\in \Zbb_+} \Pb_{k,\B, \theta}\set{\frac{1}{\psi(n)} \inf_{\vartheta\in\Gamma_{\delta,\theta}}\lambda_{\B,\vartheta}(k, k+n) 
< I_{\B,\theta}  - \varepsilon} < \infty .
\end{equation}

We ignore parameter values with $W$-measure null, i.e., without special emphasis we will always assume that 
\[
W(\Gamma_{\delta,\theta})>0 ~~ \text{for all $\theta\in\Theta$ and $\delta>0$}.
\]

The following theorem generalizes Theorem~\ref{Th:AOCMSR} to the case of the unknown post-change parameter.

\begin{theorem}\label{Th:AOCDMSR} 
Suppose that conditions \eqref{rA} and \eqref{CondPsi} hold and there exist positive and finite numbers 
$I_{\B,\theta}$, $\B\in\Pc$, $\theta\in\Theta$, such that right-tail and left-tail conditions \eqref{RTcond} and \eqref{compLeft} are satisfied.
 Then the following asymptotic as $A\to\infty$ approximation holds
\begin{equation} \label{ADDDMSR}
 \ADD_{k,\B,\theta}(T_A^{\pb, W}) = \Psi\brc{\frac{\log A}{I_{\B,\theta}}} (1+o(1))\quad  \text{for all}~ k \in \Zbb_+, ~ \B\in\Pc, \theta\in\Theta.
\end{equation}
\end{theorem}

\begin{proof}
The proof of the asymptotic lower bound (for any fixed $k\in\Zbb_+$, $\B\in\Pc$ and $\theta\in\Theta$)
\begin{equation}\label{LBDMSRAasympt}
\ADD_{k,\B,\theta}(T_A^{\pb,W}) \ge \Psi\brc{\frac{\log A}{I_{\B,\theta}}} (1+o(1)) \quad \text{as}~ A \to \infty
\end{equation}
is essentially identical to that used to establish the lower bound \eqref{LBSRAasympt} for the expected delay to detection
$\ADD_{k,\B,\theta}(T_A^{\pb,\theta})$ in the proof of Theorem~\ref{Th:AOCMSR}. So it is omitted.

To prove asymptotic approximation \eqref{ADDDMSR} it suffices to show that the lower bound \eqref{LBDMSRAasympt} is attained by $T_A^{\pb,W}$, i.e., 
\begin{equation}\label{UBTDMSR}
 \limsup_{A\to\infty} \frac{\ADD_{k, \B,\theta}(T_A^{\pb,W})}{\log A}  \le \frac{1}{I_{\B,\theta}} .
\end{equation}

Let $M_A$ be as in \eqref{MA}. Since for any $n \ge 1$,
\[
\log R_{\pb,W}(k+n) \ge  \log \Lambda_{\pb,W}(k,k+n)  \ge  \inf_{\vartheta \in \Gamma_{\delta,\theta}} \lambda_{\B,\vartheta}(k, k+n) + \log W(\Gamma_{\delta,\theta})+\log p_\B ,
\] 
using essentially the same argument as in the proof of Theorem~\ref{Th:AOCMSR} that have led to inequality \eqref{ProbkTAW} we obtain 
that for all $n \ge M_A$
\begin{align*}
 \Pb_{k, \B,\theta}\brc{T_A^{\pb,W} >  n} & \le \Pb_{k, \B,\theta}\Bigg\{\frac{1}{\psi(n)} \inf_{\vartheta \in \Gamma_{\delta,\theta}} 
 \lambda_{\B,\vartheta}(k, k+n) <  I_{\B,\theta} - \varepsilon 
 \\
 &  - \frac{1}{\psi(M_A)}  \brc{\log p_\B +\log W(\Gamma_{\delta,\theta})}\Bigg\}.
\end{align*}
Hence, for all $n \ge M_A$ and all sufficiently large  $A$ such that $|\log [p_\B W(\Gamma_{\delta,\theta})]|/\psi(M_A) < \varepsilon/2$ ,
\begin{align} \label{ProbkTWSR}
\Pb_{k, \B,\theta}\brc{T_A^{\pb,W}-k >n}  
&  \le \Pb_{k,\B,\theta}\brc{\frac{1}{\psi(n)}\inf_{\vartheta \in \Gamma_{\delta,\theta}} \lambda_{\B,\vartheta}(k, k+n)  < I_{\B,\theta}  - \frac{\varepsilon}{2}}. 
\end{align}

Lemma~A1 in Tartakovsky~\cite[Page 239]{Tartakovsky_book2020} yields the inequality (for any $k\in\Zbb_+$, $\B\in\Pc$, and $\theta\in\Theta$)
\begin{align}\label{ETAWineq}
& \Eb_{k, \B,\theta}\brcs{(T_A^{\pb,W}-k)^+}   \le 
 M_{A} + \sum_{n= M_{A}}^{\infty}    \Pb_{k, \B,\theta}\brc{T_A^{\pb,W} >  n}.
 \end{align}
Using \eqref{ETAWineq} and \eqref{ProbkTWSR}, we obtain
\begin{equation}\label{EkinequppertildeTAW}
\Eb_{k, \B,\theta}\brcs{\brc{T_A^{\pb,W}-k}^+} \le 1+\Psi\brc{\floor{\frac{\log A}{I_{\B,\theta}-\varepsilon}}}+  \Upsilon_{A}(\varepsilon/2,\B, \theta),
\end{equation}
where 
\[
 \Upsilon_{A}(\varepsilon,\B, \theta) = \sum_{n=M_A}^\infty \sup_{k\in \Zbb_+} \Pb_{k,\B, \theta}\set{\frac{1}{\psi(n)} 
 \inf_{\vartheta\in\Gamma_{\delta,\theta}}\lambda_{\B,\vartheta}(k, k+n) < I_{\B,\theta}  - \varepsilon} .
\]
Next, inequality \eqref{EkinequppertildeTAW} along with the inequality $\Pb_\infty(T_A^{\pb,W} > k) > 1- (r_A +k)/A$ (see \eqref{PFADoobineq2}) 
implies the inequality
\begin{align}
  \ADD_{k, \B,\theta}(T_A^{\pb,W})
  &\le   \frac{1+ \Psi\brc{\left\lfloor \frac{\log A}{I_{\B,\theta}-\varepsilon} \right\rfloor} + \Upsilon_{A}(\varepsilon/2,\B, \theta)}{1- (r_A +k)/A}. 
  \label{ADDupperDMSR}
 \end{align}
Since  $r_A/A\to 0$ (see condition \eqref{rA}) and, by the left-tail condition \eqref{compLeft},  $\Upsilon_A(\varepsilon/2,\B, \theta) \to 0$ as $A\to \infty$
for all $\varepsilon>0$, $\B\in\Pc$, and $\theta\in\Theta$ inequality \eqref{ADDupperDMSR} implies the asymptotic inequality
\[
 \ADD_{k, \B,\theta}(T_A^{\pb,W})\le \brc{\frac{\log A}{I_{\B,\theta} - \varepsilon}} (1+o(1)) \quad \text{as}~A \to \infty,
\]
where $\varepsilon$ can be taken arbitrarily small so that the asymptotic upper bound \eqref{UBTDMSR} follows
and the proof of the asymptotic approximation  \eqref{ADDDMSR} is complete. \qed
\end{proof}

Using asymptotic approximation \eqref{ADDDMSR} in Theorem~\ref{Th:AOCDMSR} and the lower bound  \eqref{LBkinclass} in Theorem~\ref{Th:LB},  
it is easy to prove that the mixture procedure $T_A^{\pb,W}$ is asymptotically optimal to first order as  $\alpha\to0$ in class $\classalpi$. We silently 
assume that $\psi(x)$ is either a linear or super-linear function (see \eqref{Condpsi}).

\begin{theorem}\label{Th:AoptDSR} 
Let the prior distribution of the change point satisfy condition $\CP_1$.  Assume also  that the mean value $\bar\nu=\bar\nu_\alpha$ of the prior distribution
and the head-start $r=r_\alpha$ of the statistic $R_{\pb,W}(n)$ approach infinity 
as $\alpha\to0$  with such a rate that condition \eqref{PriorSRalpha} holds.
Suppose further that there exist numbers $0<I_{\B,\theta}<\infty$ ($\B\in\Pc, \theta\in\Theta$) such that conditions 
\eqref{RTcond} and \eqref{compLeft} are satisfied.  
If threshold $A_\alpha$ is so selected that $\PFA_\pi(T_{A_\alpha}^{\pb,W}) \le \alpha$  and $\log A_\alpha \sim |\log \alpha|$ as $\alpha \to 0$, 
in particular as $A_\alpha=(r_\alpha +\bar\nu_\alpha)/\alpha$, then for all  $k\in\Zbb_+$, $\B\in\Pc$ and $\theta\in\Theta$ as $\alpha\to 0$
\begin{align}
\ADD_{k,\B,\theta}(T_{A_\alpha}^{\pb,W}) & = \Psi\brc{\frac{|\log\alpha|}{I_{\B,\theta}}}(1+o(1));
\label{FOAODMSR}
\\
 \inf_{T \in \classalpi}\ADD_{k,\B,\theta}(T)  & = \ADD_{k,\B,\theta}(T_{A_\alpha}^{\pb,W}) (1+o(1)),
 \label{FOAODMSR2}
 \end{align}
that is, the procedure $T_{A_\alpha}^{\pb,W}$ is first-order asymptotically optimal as $\alpha\to0$ in class $\classalpi$.
\end{theorem}

\begin{proof}
If $A_\alpha$ is so selected that  $\log A_\alpha\sim |\log\alpha|$ as $\alpha\to0$, then asymptotic approximation \eqref{FOAODMSR}
follows from asymptotic approximation \eqref{ADDDMSR} in Theorem~\ref{Th:AOCDMSR}.  Since this approximation is the same as the 
asymptotic lower bound \eqref{LBkinclass} in Theorem~\ref{Th:LB}, this shows that the lower bound is attained by 
the detection rule $T_{A_\alpha}^{\pb,W}$, so \eqref{FOAODMSR2} follows and the proof is complete.

If, in particular,  $A_\alpha=(r_\alpha +\bar\nu_\alpha)/\alpha$, then $\log A_\alpha\sim |\log \alpha|$ and by Lemma~\ref{Lem:PFADSR}  
$T_{A_\alpha}^{\pb,W}\in \classalpi$. \qed
\end{proof}

\begin{remark}\label{Rem:DMSR}
The assertion of Theorem~\ref{Th:AoptDSR} holds when $\bar{\nu}<\infty$ and $r$ do not depend on $\alpha$ and
the condition $\CP_1$ is satisfied with $\beta \equiv 0$, i.e., for priors $\pi$ with heavy tails. This could be expected since the detection statistic 
$R_{\pb,W}(n)$ uses an improper uniform prior distribution of the change point on the whole positive line. 
\end{remark}

\subsubsection{The Case of Independent Streams}\label{sssec:indstreams}

Notice that so far we considered a very general stochastic model where not only the observations in streams may be dependent and 
non-identically distributed, but also the streams may be mutually dependent. In this very general case, computing statistic $R_{\pb,W}(n)$ 
is problematic even when the statistics in data streams can be computed.  Consider now still a very general scenario where the 
data streams are mutually independent (but have a general statistical structure), which is of special interest for many applications. 
The computational problem becomes more manageable when the data between data streams are independent. 

\paragraph{\bf \ref{sssec:indstreams}(a) Computational Issues}

In the case where the data across streams are independent, the model has the form
\begin{equation}\label{ind}
\begin{split}
& p(\Xb^n | H_{k,\B}, \teb_\B)  =  \prod_{t=1}^n \prod_{i=1}^N g_i(X_t(i)|\Xb^{t-1}(i)) \quad \text{for}~ \nu=k \ge n ,
\\
& p(\Xb^n | H_{k,\B},\teb_\B)  =  \prod_{t=1}^{k}  \prod_{i=1}^N g_i(X_t(i)|\Xb^{t-1}(i)) \times 
\\
& \quad \prod_{t=k+1}^{n}  \prod_{i\in \B} f_{i,\theta_i}(X_t(i)|\Xb^{t-1}(i)) \prod_{i\notin \B} g_i(X_t(i)|\Xb^{t-1}(i)) 
\quad \text{for}~ \nu=k < n,
\end{split}
\end{equation}
where $g_i(X_t(i)|\Xb^{t-1}(i))$ and $f_{i,\theta_i}(X_t(i)|\Xb^{t-1}(i))$ are conditional pre- and post-change densities in the $i$-th data stream, 
respectively, $\theta_i\in\Theta_i$ is the unknown post-change parameter (generally multidimensional) in the $i$-th stream ($i\in \{1,\dots,N\}=\Nc$), 
and $\teb_\B=(\theta_i,i\in\B)$ is the vector of parameters in the set $\B$. So the LR processes are
\begin{equation}\label{LRind}
LR_{\B,\teb_\B}(k, n) = \prod_{i\in \B} LR_{i,\theta_i}(k,n), \quad LR_{i,\theta_i}(k,n)= \prod_{t=k+1}^{n}  \Lc_{i,\theta_i}(t), \quad n > k,
\end{equation}
where $\Lc_{i,\theta_i}(t)=f_{i,\theta_i}(X_t(i)|\Xb^{t-1}(i))/g_i(X_t(i)|\Xb^{t-1}(i))$.

Recall that $\Pc_K =\{\B: 1 \le |\B| \le K\}$ is the subclass of $\Pc$ for which the cardinality $|\B|$ of the sets where the change may occur does not exceed $K \le N$
streams, and by $\Pc_K^\star =\{\B:  |\B| = K\}$ denote the subclass of $\Pc$ for which the change occurs in exactly $K$ streams. 

Assume, in addition, that the mixing measure is such that
\[
p_\B= C(\Pc_K) \prod_{i\in \B} p_i, \quad C(\Pc_K) = \brc{\sum_{\B \in \Pc_K} \prod_{i \in \B} p_i}^{-1}. 
\]
Then the mixture LR is
\[
\Lambda_{\pb,\teb}(k,n) = C(\Pc_K) \sum_{i=1}^K \sum_{\B\in \Pc_i^\star} \prod_{j\in\B} p_j LR_{j,\theta_i}(k,n),
\]
and its computational complexity is polynomial in the number of data streams. Moreover, in the special, most difficult case of $K=N$ and $p_j=p$, we obtain
\begin{equation}\label{mixLRind}
\Lambda_{\pb,\teb}(k,n) = C(\Pc_N) \brcs{\prod_{i=1}^N \brc{1+ p LR_{i,\theta_i}(k,n)} -1},
\end{equation}
so its computational complexity is only $O(N)$. The representation \eqref{mixLRind} corresponds to the case when each stream is affected independently with probability $p/(1+p)$,
the assumption that was made in \cite{Xie&Siegmund-AS13}.

\paragraph{\bf \ref{sssec:indstreams}(b) Asymptotic Optimality of Detection Procedures}

Since the data are independent across streams, for an assumed
value of the change point $\nu = k$,  stream $i \in \Nc$, and the post-change parameter value in the $i$-th stream $\theta_i$, the LLR of observations accumulated by time $k+n$ is given by
\[
\lambda_{i,\theta_i}(k, k+n)= \sum_{t=k+1}^{k+n}\log \frac{f_{i,\theta_i}(X_t(i)| \Xb^{t-1}(i))}{g_{i}(X_t(i))| \Xb^{t-1}(i))}, \quad n \ge 1.
\]
Define $\Gamma_{\delta,\theta_i}=\{\vartheta\in\Theta_i\,:\,\vert \vartheta-\theta_i\vert<\delta\}$,
\begin{align*}
\gamma_{k,L}^{(i)}(\varepsilon, \theta_i) & = 
\Pb_{k,i,\theta_i}\set{\frac{1}{\psi(L)}\max_{1 \le n \le L} \lambda_{i,\theta_i}(k, k+n) \ge (1+\varepsilon) I_{i,\theta_i}} ,
\\
\Upsilon^{(i)}(\varepsilon, \theta_i) & =   \lim_{\delta\to0}\sum_{n=1}^\infty \sup_{k\in \Zbb_+} 
\Pb_{k,i, \theta_i}\set{\frac{1}{\psi(n)} \inf_{\vartheta\in\Gamma_{\delta,\theta_i}}\lambda_{i,\vartheta}(k, k+n) < I_{i,\theta_i}  - \varepsilon},
\end{align*}
and assume that the following right-tail and left-tail conditions are satisfied for local LLR statistics in data streams:
There exist positive and finite numbers $I_{i,\theta_i}$, $\theta_i\in \Theta_i$, $i \in \Nc$, such that for any  $\varepsilon >0$
\begin{equation}\label{Pmaxi}
\lim_{L\to\infty} \gamma_{k,L}^{(i)}(\varepsilon, \theta_i)  =0 \quad \text{for all}~ k\in \Zbb_+, ~ \theta_i\in\Theta_i,  ~ i \in \Nc  ;
\end{equation}
and
\begin{equation}\label{rcompLefti} 
\Upsilon^{(i)}(\varepsilon, \theta_i) < \infty \quad \text{for all}~  \theta_i\in\Theta_i, ~ i\in \Nc  .
\end{equation}

We also assume that $W(\Gamma_{\delta,\theta_i})>0$ for all $\delta>0$ and $i\in\Nc$.

Let $I_{\B,\teb_{\B}} = \sum_{i\in\B} I_{i,\theta_i}$. Recall that
\[
\gamma_{k,L}(\varepsilon,\B,\theta)  = 
\Pb_{k, \B,\theta}\set{\frac{1}{\psi(L)}\max_{1 \le n \le L} \lambda_{\B,\theta}(k, k+n) \ge (1+\varepsilon) I_{\B,\theta}} .
\]
Since the LLR process $\lambda_{\B,\teb_\B}(k,k+n)$ is the sum of independent local LLRs,
$\lambda_{\B,\teb_\B}(k,k+n) = \sum_{i\in\B} \lambda_{i,\theta_i}(k,k+n)$ (see \eqref{LRind}), it is easy to see that
\[
\gamma_{k,L}(\varepsilon,\B,\teb_\B) \le \sum_{i\in\B}  \gamma_{k,L}^{(i)}(\varepsilon, \theta_i) ,
\]
so that local conditions \eqref{Pmaxi} imply global right-tail condition \eqref{RTcond}. This is true, in particular, if  
$\lambda_{i,\theta_i}(k, k+n)/\psi(n)$ converge $\Pb_{k,i, \theta_i}$-a.s.\ to  $I_{i,\theta_i}$,  $i=1,\dots,N$, in which case the 
SLLN for the global LLR \eqref{SLLNLLR} holds with $I_{\B,\teb_{\B}} = \sum_{i\in\B} I_{i,\theta_i}$.
Also,
\[
\Upsilon(\varepsilon, \B,\teb_\B) \le  \sum_{i\in\B} \Upsilon^{(i)}(\varepsilon, \B,\theta_i) ,
\]
which shows that local left-tail conditions \eqref{rcompLefti} imply global left-tail condition \eqref{compLeft}.

Theorem~\ref{Th:AoptDSR} implies the following results on asymptotic properties of the mixture procedure $T_A^{\pb,W}$.

\begin{corollary}\label{Cor:Cor1}
Assume that for some positive and finite numbers $I_{i,\theta_i}$, $\theta_i\in\Theta_i$, $i=1\dots,N$, 
right-tail and left-tail conditions \eqref{Pmaxi} and  \eqref{rcompLefti} for local data streams are satisfied. 
If threshold $A_\alpha$ is so selected that $\PFA_\pi(T_{A_\alpha}^{\pb,W}) \le \alpha$  and $\log A_\alpha \sim |\log \alpha|$ as $\alpha \to 0$, 
in particular as $A_\alpha=(r_\alpha +\bar\nu_\alpha)/\alpha$, and if conditions \eqref{Condpsi} and \eqref{PriorSRalpha} are satisfied, then asymptotic 
formulas \eqref{FOAODMSR} and  \eqref{FOAODMSR2} hold with $I_{\B,\theta} = I_{\B,\teb_\B}=\sum_{i\in\B} I_{i,\theta_i}$.
Therefore,  detection procedure $T_{A_\alpha}^{\pb,W}$ is first-order asymptotically optimal as $\alpha\to0$ in class $\classalpi$.
\end{corollary}

\begin{remark} \label{Rem: Wi}
The assertions of Corollary~\ref{Cor:Cor1} also hold for different distribution functions $W_i$, $i\in \Nc$ in streams if we assume that $W_i(\Gamma_{\delta,\theta_i})>0$ for all $\delta>0$ and $i\in\Nc$. 
A modification in the proof is trivial.
\end{remark}

\begin{remark} \label{Rem: theta}
In the case where the post-change parameter is known, essentially similar corollary follows from Theorem~\ref{Th:FOAOMSR} if we require the uniform complete version of the SLLN for all streams $i\in \Nc$:
\begin{equation}\label{Ucompi}
\sum_{n=1}^\infty \sup_{k\in\Zbb_+} \Pb_{k, i,\theta_i}\set{\abs{\frac{\lambda_{i,\theta_i}(k,k+n)}{\psi(n)} - I_{i,\theta_i}}> \varepsilon} < \infty 
\quad \text{for all}~ \varepsilon >0.
\end{equation}
\end{remark}

\section{Examples}\label{sec:Examples}

\subsection{Detection of Changes in the Mean Values of Multistream Autoregressive Non-stationary Processes}\label{ssec:Ex1}

This example related to detecting changes in unknown means of multistream non-stationary AR($p$) processes has a specific real application in 
addition to several other applications in mathematical statistics. Specifically, 
it arises in multi-channel sensor systems 
(such as radars  and electro-optic imaging systems) where it is required to detect an  unknown number of randomly appearing signals from
objects in clutter and sensor noise (cf., e.g., \cite{Bakutetal-book63,Tartakovsky&Brown-IEEEAES08,TNB_book2014}). 

Observations in the $i$-th channel have the  form
\begin{equation}\label{MeanARmodel}
X_{n}(i)=\theta_i \, S_{n}(i) \Ind{n > \nu}  +\xi_{n}(i),\quad n \ge 1, ~ i =1,\dots,N ,
\end{equation}
where $\theta_i \, S_{n}(i)$ are deterministic signals with unknown ``amplitudes'' $\theta_i>0$ that appear at an unknown time $\nu$
in additive noises $\xi_{n}(i)$ in an $N$-channel system.  
Assume that noise processes $\{\xi_{n}(i)\}_{n \in\Zbb_+}$ ($i=1,\dots,N$) are mutually independent  $p$-th order Gaussian autoregressive processes 
AR$(p)$ that obey recursions
\begin{equation}\label{sec:Ex.1}
\xi_{n}(i) = \sum_{j=1}^p \rho_{i,j} \xi_{n-j}(i) + w_{n}(i), \quad n \ge 1, 
\end{equation}
where   $\{w_{n}(i)\}_{n\ge 1}$, $i=1,\dots,N$, are mutually independent i.i.d.\ normal $\Nc(0,\sigma_i^2)$ sequences ($\sigma_i>0$), so the observations 
in channels $X_{n}(1),\dots,X_{N}(n)$ are independent of each other. For simplicity, let us set zero initial conditions $\xi_{1-p}(i)=\xi_{2-p}(i)=\cdots=\xi_{0}(i)=0$. 
The coefficients $\rho_{i,1},\dots,\rho_{i,p}$ and variances $\sigma_i^2$ are known and all roots of the equation $z^p -\rho_{i,1} z^{p-1} - \cdots - \rho_{i,p}=0$ 
are in the interior of the unit circle, so that the AR($p$) processes are stable.  

For $n \ge 1$, define the $p_n$-th order residuals 
\[
\widetilde{S}_{n}(i) = S_{n}(i)- \sum_{j=1}^{p_n} \rho_{i,j} S_{n-j}(i), \quad  \widetilde{X}_{n}(i) = X_{n}(i)- \sum_{j=1}^{p_n} \rho_{i,j} X_{n-j}(i),
\]
where $p_n =p$ if $n > p$ and $p_n =n$ if $n \le p$.  It is easily shown that the conditional pre-change and post-change densities in the $i$-th channel are
\begin{align*}
g_i(X_{n}(i) |\Xb^{n-1}(i))& = f_{0,i}(X_{n}(i)|\Xb^{n-1}(i))= \frac{1}{\sqrt{2\pi \sigma_i^2}} \exp\set{-\frac{\wtX_{n}(i)^2}{2\sigma_i^2}},
\\
f_{i,\theta_i}(X_{n}(i) |\Xb^{n-1}(i))&=   \frac{1}{\sqrt{2\pi \sigma_i^2}} \exp\set{-\frac{(\wtX_{n}(i)-\theta_i \wtS_{n}(i))^2}{2\sigma_i^2}} , \quad
\theta_i \in (0, \infty),
\end{align*}
and that for all $k \in \Zbb_+$ and $n \ge 1$  the LLR in the $i$-th channel has the form
$$
\lambda_{i,\theta_i}(k, k+n) = \frac{\theta_i}{\sigma_i^2}  \sum_{t=k+1}^{k+n} \wtS_{t}(i) \wtX_{t}(i) -
\frac{\theta_i^2 \sum_{t=k+1}^{k+n} \wtS_{t}(i)^2}{2 \sigma_i^2} .
$$
Since under measure $\Pb_{k,i,\vartheta_i} $ the random variables $\{\wtX_{n}(i)\}_{n\ge k+1}$ are independent Gaussian random variables 
$\Nc(\vartheta_i \wtS_{n(i)},\sigma_i^2)$, under  $\Pb_{k,i,\vartheta_i} $ the LLR $\lambda_{i,\theta_i}(k, k+n) $ is a Gaussian process 
(with independent but non-identically distributed increments) with mean and variance
\begin{equation}\label{LLRAR1}
\begin{split}
\Eb_{k,i,\vartheta_i} [\lambda_{i, \theta_i}(k, k+n)]  & = \frac{2 \theta_i \vartheta_i -\theta_i^2}{2\sigma_i^2} \sum_{t=k+1}^{k+n} \wtS_{t}(i)^2 , 
\\
\Var_{k,i,\vartheta_i} [\lambda_{i, \theta_i}(k, k+n)]  & =   \frac{\theta_i^2}{\sigma_i^2} \sum_{t=k+1}^{k+n} \wtS_{t}(i)^2 .
\end{split}
\end{equation}

Assume that 
\begin{equation*}
\lim_{n\to \infty} \frac{1}{\psi(n)}  \sup_{k \in \Zbb_+} \sum_{t=k+1}^{k+n} \wtS_{t}(i)^2 = Q_i ,
\end{equation*}
where  $0<Q_i <\infty$. In a variety of signal processing applications this condition holds with $\psi(n)=n$, e.g., in radar applications where the 
signals $\theta_i \, S_{i,n}$ are the sequences of harmonic pulses. In some applications such as detection, recognition, and tracking of objects on ballistic
trajectories that can be approximated by polynomials of order $m=2-3$, the function $\psi(n)=n^m$, $m >1$. 
Then for all $k\in\Zbb_+$ and $\theta_i\in (0,\infty)$
\[
\frac{1}{\psi(n)}\lambda_{i,\theta_i}(k, k+n)  \xra[n\to\infty]{ \Pb_{k,i,\theta_i} -\text{a.s.}} \frac{\theta_i^2 Q_i}{2\sigma_i^2} =I_{i,\theta_i},
\]
so that the SLLN \eqref{SLLNLLR} takes place, and therefore, the right-tail condition \eqref{Pmaxi} holds. 
Furthermore, since the second moment of the LLR is finite it can be shown that uniform complete convergence conditions \eqref{Ucompi} as well as
the left-tail condition \eqref{rcompLefti} also hold. 
Thus, by Corollary~\ref{Cor:Cor1}, the mixture detection procedure $T_{A_\alpha}^{\pb,W}$ minimizes as $\alpha\to0$ the expected delay to detection and
asymptotic formulas \eqref{FOAODMSR} and  \eqref{FOAODMSR2} hold with 
\[
I_{\B,\theta}=I_{\B,\teb_\B}=\sum_{i\in\B}  \frac{\theta_i^2 Q_i}{2\sigma_i^2}.
\]

\subsection{Change Detection in the Spectrum of the AR($p)$ Multistream Process}\label{ssub:Ex2}
 
Consider the problem of detecting the change of the correlation coefficient in the $p$-th order AR process which in the $i$-th stream satisfies the recursion 
\begin{equation}\label{sec:Ex.7}
X_{n}(i) =\sum_{j=1}^{p} \rho^{(n)}_{i,j}\,X_{n-j}(i)+w_{n}(i) \, , \quad n\ge 1,
\end{equation}
where 
\[
\rho^{(n)}_{i,\ell}=\theta^{*}_{i,\ell}\Ind{n \le \nu}+\theta_{i,\ell}\Ind{n > \nu}
\]
and  $\{w_{n}(i)\}_{n\ge 1}$ are i.i.d. (mutually independent) Gaussian random variables with $\EV[w_{1}(i)]=0$, $\EV[w^{2}_{1}(i)]=1$.  
Additional notation:  
\[
\theta^{*}_{i}=(\theta^{*}_{i,1},\ldots,\theta^{*}_{i,p})^\top, ~~ \theta_{i}=(\theta_{i,1},\ldots,\theta_{i,p})^\top, ~~ 
\Xb_i^{n-1,n-p}= (X_{n-1}(i),\ldots,X_{n-p}(i)).
\]
($\top$ denotes transpose).

It is easy to see that the pre-change and post-change conditional densities $g_i(X_{i,n}\vert \Xb_i^{n-1})=g_i(X_{n}(i)\vert  \Xb_i^{n-1,n-p})$ 
and  $f_{i, \theta_i}(X_{n}(i)\vert \Xb_i^{n-1}(i))= f_{i,\theta_i}(X_{n}(i)\vert \Xb_i^{n-1,n-p})$ are given by 
 \begin{equation}
 \label{dens-AR-p-1}
 \begin{split}
 g_i(X_{n}(i)\vert  \Xb_i^{n-1,n-p}) & = \frac{1}{(2\pi)^{p/2}} \,\exp\set{-\dfrac{(\eta^{*}_{i}(X_n(i),  \Xb_i^{n-1,n-p})^{2}}{2}} ,
 \\
 f_{i,\theta_i}(X_{n}(i)\vert \Xb_i^{n-1,n-p}) & = \frac{1}{(2\pi)^{p/2}} \,\exp\set{-\dfrac{(\eta_{\theta_i,i}(X_n(i),  \Xb_i^{n-1,n-p}))^{2}}{2}},
 \end{split}
\end{equation}
where $\eta^{*}_{i}(y,x)=y-(\theta^{*}_{i})^{\top} x$ and $\eta_{\theta_i,i}(y,x)=y-(\theta_{i})^{\top} x$ ($x\in\bbr$, $y\in \bbr^p$).
Therefore,  for any $\theta_i\in\bbr^{p}$, the LLR is 
\[
\lambda_i(k, k+n) = \sum_{t=k+1}^{k+n}  \lambda_i^*(t)
\]
where
\begin{equation}\label{func-g-Coef-nn-1}
\begin{split}
\lambda_i^*(t) & =
\log\frac{f_{i,\theta_i}(X_t(i)\vert \Xb_i^{t-1,t-p})}{g_{i}(X_t(i)\vert \Xb_i^{t-1,t-p})}
\\
& =
X_t(i)(\theta_i-\theta^{*}_{i})^{\top} \Xb_i^{t-1,t-p}+\frac{((\theta^{*}_{i})^{\top} \Xb_i^{t-1,t-p})^{2}-(\theta_i^{\top}\Xb_i^{t-1,t-p})^{2}}{2}.
\end{split}
\end{equation}
 The process \eqref{sec:Ex.7} is not Markov, but the $p$-dimensional processes
 \[
\Phi_{i,n}=(X_{i,n},\ldots,X_{i,n-p+1})^{\top}\in\bbr^{p},  ~~ i =1,\dots,N
\]
are Markov. Now, for any $\vartheta=(\vartheta_{1},\ldots,\vartheta_{p})\in\bbr^{p}$, define the matrix
\[
\Lambda(\vartheta)= \begin{pmatrix}
\vartheta_{1} & \vartheta_2 & \dots & \vartheta_{p}\\
1 & 0 &  \dots & 0\\
\vdots & \vdots &  \ddots & \vdots \\
0 & 0 &\dots 1&0 
\end{pmatrix} \, .
\]
Note that
\begin{equation}\label{Phi}
\Phi_{i,n} = \begin{cases}
 \Lambda(\theta_i^*) \Phi_{i,n-1}+\wt{w}_{i,n}  &  \mbox{for} ~ n\le \nu ,
 \\
\Lambda(\theta_i)\Phi_{i,n-1}+\wt{w}_{i,n} &  \mbox{for} ~ n >  \nu ,
\end{cases}
\end{equation}
where 
$\wt{w}_{i,n}=(w_{i,n},0,\ldots,0)^\top \in\bbr^{p}$.
Obviously,
\[
\EV[\wt{w}_{n}\,\wt{w}^{\top}_{n}] =B=
\begin{pmatrix}
1 & \dots & 0\\
\vdots & \ddots & \vdots\\
0 & \dots & 0
\end{pmatrix} .
\]

Assume that all eigenvalues of the matrices $\Lambda(\theta^{*}_{i})$ in modules are less than $1$ and that $\theta_{i}$ 
belongs to the set 
\begin{equation}
\label{set-Theta-i}
\Theta^{st}_{i}=\{\theta_i\in \bbr^{p}\,:\, \max_{1\le j\le p}\,\vert\e_{j}(\Lambda(\theta_i))\vert<1\}\setminus \,\{\theta^{*}_{i}\},
\end{equation}
where $\e_{j}(\Lambda)$ is the $j$-th eigenvalue of the matrix $\Lambda$.
Using \eqref{Phi} it can be shown that in this case the processes $\{\Phi_{i,n}\}_{n>\nu}$ ($i =1,\dots,N$) are ergodic with stationary 
normal distributions $\Nc(0,\F_{i}(\theta_i))$, where 
\[
\F_{i}(\theta_i)=\sum_{n\ge 0} (\Lambda(\theta_i))^n B (\Lambda^\top(\theta_i))^{n}.
\]

Taking into account that 
$\max_{1\le i\le N}\,\sup_{t\ge  1}\,\Eb_\infty\vert X_{t}(i)\vert^{m}<\infty$
for any $m>0$ and using techniques developed in \cite{PerTar-JMVA2019} it can be shown that the uniform complete convergence conditions
\eqref{Ucompi} as well as the left-tail conditions \eqref{rcompLefti} hold. Therefore, Corollary~\ref{Cor:Cor1}
implies that the mixture detection procedure $T_{A_\alpha}^{\pb,W}$ minimizes as $\alpha\to0$ the expected delay to detection for any compact subset 
of $\Theta=\Theta^{st}_{1}\times\cdots\times \Theta^{st}_{N}$ and
asymptotic formulas \eqref{FOAODMSR} and  \eqref{FOAODMSR2} hold with
\[
I_{\B,\theta}=I_{\B,\teb_\B}= \frac{1}{2} \sum_{i\in\B}  (\theta_i-\theta_i^*)^\top \F_{i}(\theta_i)(\theta_i-\theta_i^*) .
\] 
  In particular, in the purely Markov scalar case where $p=1$ in \eqref{sec:Ex.7},
 \[
I_{\B,\theta}=I_{\B,\teb_\B}= \sum_{i\in\B}  \frac{(\theta_i-\theta^{*}_{i})^{2}}{2(1-\theta_i^{2})} .
\] 
 
\section{Monte Carlo} \label{sec:MC} 

In this section, we perform MC simulations for the example considered in Subsection~\ref{ssec:Ex1}, assuming for simplicity (and with a very minor
loss of generality) that the noise processes $\xi_n(i)=w_n(i)$ in \eqref{MeanARmodel} are i.i.d. Gaussian with mean zero and variances 
$\sigma_i^2$, which is equivalent setting $p=0$ in \eqref{sec:Ex.1}. We also assume that $S_n(i) = n^{1.1}$, $\sigma_i^2=4$, $\theta_i=0$ pre-change and 
$\theta_i=\theta=0.1$ in the post-change mode when the change occurs in the $i$-th stream.

We suppose that the change in each channel occurs independently with probability $q$, so the probability of the event $\{M=m\}$
that $m$ streams (out of $N$) are affected is  
\[
\Pb(M=m) = \frac{N!}{m! (N - m)!} \, q^{m} \, (1 - q)^{N - m} .
\]
The change occurs at time $\nu$ according to the geometric prior distribution 
\[
\Pb(\nu=k)=\varrho(1-\varrho)^k, \quad k=0,1\dots
\] 
with a parameter $\varrho \in (0,1)$. 

In MC simulations, we set $N=10$ for the total number of streams, $\varrho=0.1$, $q=1/N=0.1$, and consider three cases: 
(a) the change occurs in a single stream with $\theta=0.1$, 
(b) the change occurs in two streams with $\theta=0.1$ in each, and (c)  the change occurs in three streams with $\theta=0.1$ in each.

For each MC run, using equations \eqref{LRind} and \eqref{mixLRind}, we compute the statistics \eqref{MSR_stat} and \eqref{DMSR_stat} and get the stopping times 
for the cases where the parameter of the post-change distribution is known \eqref{MS_def} and when the parameter of the post-change distribution 
is unknown \eqref{DMSR_def}. We take uniform prior $W(\theta)$ on $[0.1,0.3]$ with the step $0.01$. The number of Monte Carlo runs is $10^6$, 
which ensures very high accuracy of the estimated characteristics.

The results are shown in Table~\ref{t:ADD_vs_PFA} and Fig.~\ref{EDD_vs_PFA}. In Fig.~\ref{EDD_vs_PFA}, $\PFA$ (x-axis) is presented in a logarithmic scale. Marks on the figure $2, 3, \dots, 9$ mean how many times the $\PFA$ value is greater than the previous decimal mark (0.0001, 0.001, or 0.01).
It is seen that the detection algorithm has good performance even with such a low 
signal-to-noise ratio. As expected, in the case when the parameter of the post-change distribution is not known, the algorithm works worse than when it is known, 
but only slightly -- the difference is small. When the change occurs in multiple streams,  the expected detection delays  
$\ADD_{\B_2,\theta}(T_A^{\pb,W})$ and $\ADD_{\B_3,\theta}(T_A^{\pb,W})$ decrease compared to $\ADD_{\B_1,\theta}(T_A^{\pb,W})$ in the case (a), and also
$\ADD_{\B_3,\theta}(T_A^{\pb,W})$ becomes smaller than $\ADD_{\B_2,\theta}(T_A^{\pb,W})$, as expected from theoretical results (see, e.g.,
\eqref{ADDDMSR}). Here $\ADD_{\B_m,\theta}(T_A^{\pb,W})$ denotes the expected detection delay when the change occurs in $m$ streams ($m=1,2, 3$).

It is also interesting to compare the MC estimates of the expected detection delay with theoretical asymptotic approximations given by \eqref{ADDDMSR}.
In our example, these approximations reduce to
\begin{equation*} 
\overline{ \ADD}_{\B_m,\theta} \approx \brc{\frac{\log A_m}{I_{\B_m,\theta}}}^{1/3.2}
\end{equation*}
with $I_{\B_m,\theta}=m (\theta^2/6.4 \sigma^2)$, $m=1,2,3$. Note that thresholds $A_m$ are different for different $m$ to guarantee the same PFA. 
The data in Table~\ref{t:ADD_vs_PFA} show that the first-order asymptotic approximation is not too bad but not especially accurate.

\begin{table}[h]
	\begin{center}
    \newcommand{\ten}[2]{$#1 \cdot\! 10^{#2}$}
    \caption{Operating characteristics of the mixture detection procedures for the number of streams $N=10$ with parameters $S_n(i) = n^{1.1}$ and 
    $\theta_1=\theta_2=\theta_3=\theta=0.1$. The number of MC runs is $10^6$.}
    \renewcommand{\arraystretch}{1.35}
    \begin{tabular}{| c  | l  |c | c |  c | c |c |c |c|}
    \hline\hline
        $\PFA_\pi(T)$ & $10^{-1}$ & \ten{5}{-2} & $10^{-2}$ & \ten{5}{-3} & $10^{-3}$ & \ten{5}{-4} \\
        \hline
	$\ADD_{\B_1,\theta}(T_A^{\pb,\theta})$    	& 10.77 & 11.28 & 12.76 & 13.34 & 14.68 & 15.22 \\       
        $\ADD_{\B_1,\theta}(T_A^{\pb,W})$       	& 11.33 & 12.07 & 13.32 & 13.93 & 15.55 & 16.02 \\ 
        $\overline{\ADD}_{\B_1,\theta}$			& 15.99 & 16.92 & 18.84 & 19.50 & 21.05 & 21.63 \\       
        \hline
        $\ADD_{\B_2,\theta}(T_A^{\pb,\theta})$ 	& 8.77 & 9.33 & 10.50 & 10.88 & 12.00 & 12.46 \\       
        $\ADD_{\B_2,\theta}(T_A^{\pb,W})$      	& 8.90 & 9.52 & 10.82 & 11.30 & 12.34 & 12.81 \\
        $\overline{\ADD}_{\B_2,\theta}$			& 11.78 & 12.66 & 14.43 & 15.02 & 16.38 & 16.88 \\
        \hline
        $\ADD_{\B_3,\theta}(T_A^{\pb,\theta})$ 	& 8.01 & 8.34 & 9.33 & 9.79 & 10.77 & 11.12 \\       
        $\ADD_{\B_3,\theta}(T_A^{\pb,W})$         	& 8.05 & 8.47 & 9.55 & 9.99 & 10.96 & 11.39 \\
        $\overline{\ADD}_{\B_3,\theta}$		        & 7.90 & 9.20 & 11.35 & 12.00 & 13.45 & 13.96 \\ \hline\hline
	\end{tabular}
    \label{t:ADD_vs_PFA}
    \end{center}
\end{table}

\begin{figure}[!h!] \centering
\includegraphics[width=\textwidth]{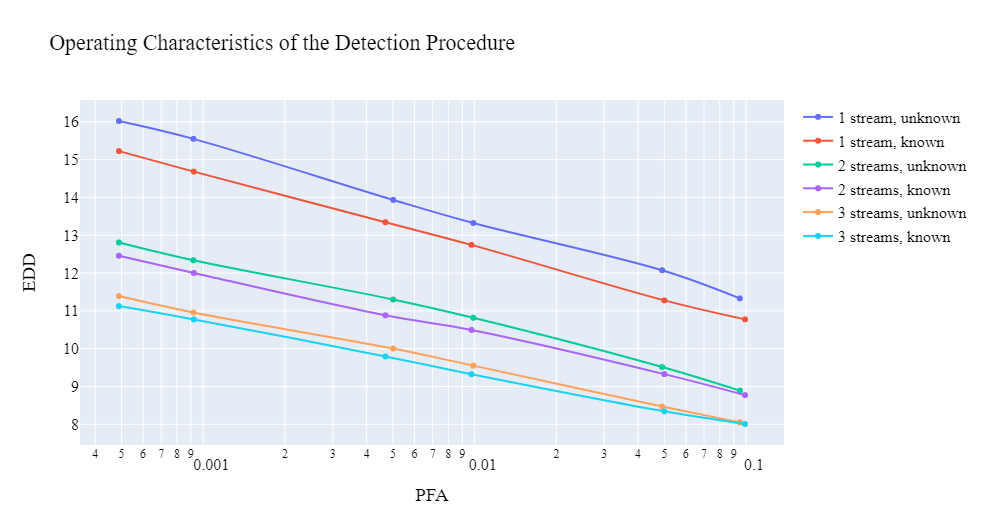}
\caption{Operating characteristics of the detection procedures for the total number of streams $N=10$ and for the number of affected streams
$m=1,2,3$: $\ADD_{\B_m,\theta}(T_A^{\pb,\theta})$ vs $\log \PFA_\pi$ and $\ADD_{\B_m,\theta}(T_A^{\pb,W})$ vs $\log \PFA_\pi$. Number of MC runs $10^6$. x-axis in logarithmic scale}  
\label{EDD_vs_PFA}
\end{figure}

\section{Applications} \label{sec:COVID} 

In this section, we consider two different applications -- rapid detection of new COVID-19 epidemic waves and extraction of tracks of low-observable near-Earth 
space objects in optical images obtained by telescopes.

\subsection{Application to COVID Detection} \label{ssec:COVID} 

We consider the problem of detecting the emergence of new COVID-19 epidemic waves in Australia based on real data. 
Note that in \cite{pergamenchtchikov2022minimax} the algorithm of joint disorder detection and identification has been used not only to detect the start of COVID-19 in Italy 
but also to identify a region where the outbreak occurs. Here we do not consider identification of the region in 
which the outbreak occurs. We need only to decide on the occurrence of an epidemic in the whole country based on data from various regions. 
Also, in contrast to \cite{pergamenchtchikov2022minimax} where a stationary Markov model has been used, we now propose a substantially different non-stationary 
model taking into account that a new wave of COVID typically spreads faster than a linear law. This fact has been recently noticed in  \cite{LiangTarVeerIEEEIT2023}.

We selected data on COVID in Australia, which has 8 regions: Australian Capital Territory, New South Wales, Northern Territory, Queensland, 
South Australia, Tasmania, Victoria, Western Australia. We propose to use the non-stationary model \eqref{MeanARmodel} considered in Subsection~\ref{ssec:Ex1}
with super-linear functions $S_n(i)= C_i \, n^{\gamma}$, $\gamma>1$. 
As an observation, we use the percentage of infections in the total population. We study the COVID outbreak in Australia, 
which was recorded in the winter of 2021-2022. The data are taken from the World Health Organization and presented in Fig.~\ref{Dist_COVID}.

\begin{figure}[htb] \centering
\includegraphics[width=\textwidth]{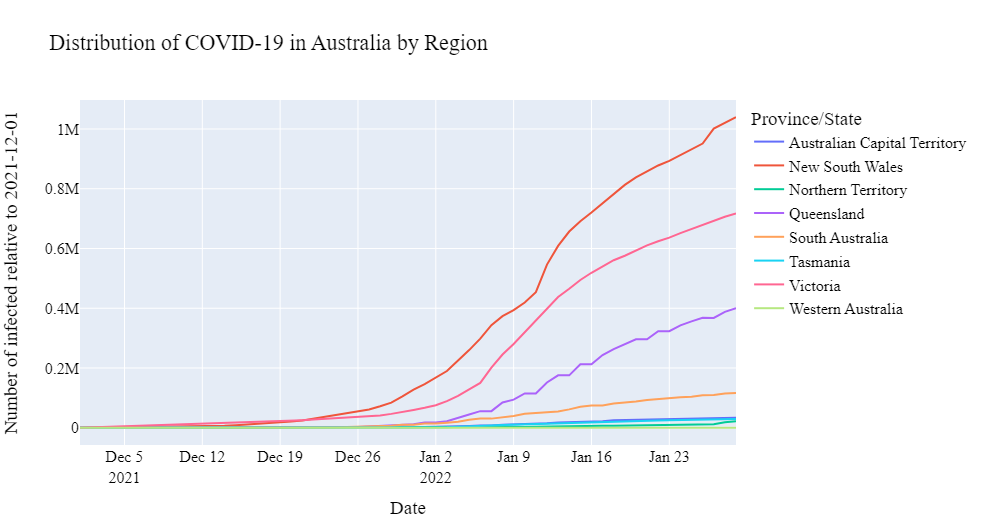}
\caption{Distribution of COVID-19 in Australia by region.}  
\label{Dist_COVID}
\end{figure}

A performed statistical analysis shows that the model with parameters $\theta_1=0.1$ and $\theta_2=0.08$ for New South Wales and Victoria, respectively, 
describes the beginning of the pandemic outbreak well. 
We also selected $S_n(i) = n^{1.127}$ and $\xi_n(i)=w_n(i)\sim \Nc(0,\sigma_i^2)$ with $\sigma_i = 2.4$ in \eqref{MeanARmodel}. We apply the $8$-stream 
double mixture detection algorithm discussed above when the 
parameter of the post-change distribution is not known. We used the discrete uniform prior $W(\theta)$ on $[0.08, \dots, 0.11]$ with step $0.01$. 
The proposed change detection algorithm with a threshold selected so that the probability of false detection does not exceed 0.01 (average detection delay is about 10) 
decides on the outbreak of the epidemic on January 9, 2022 (see Fig.~\ref{Detection_stat}). 

The plots in Fig.~\ref{Detection_stat} illustrate the behavior of the change detection statistic $R_{\pb,W}(n)$ defined in \eqref{DMSR_stat}. 
The COVID wave is detected at the moment when the statistic $R_{\pb,W}(n)$ crosses the threshold.

\begin{figure}[!ht] \centering
\includegraphics[width=\textwidth]{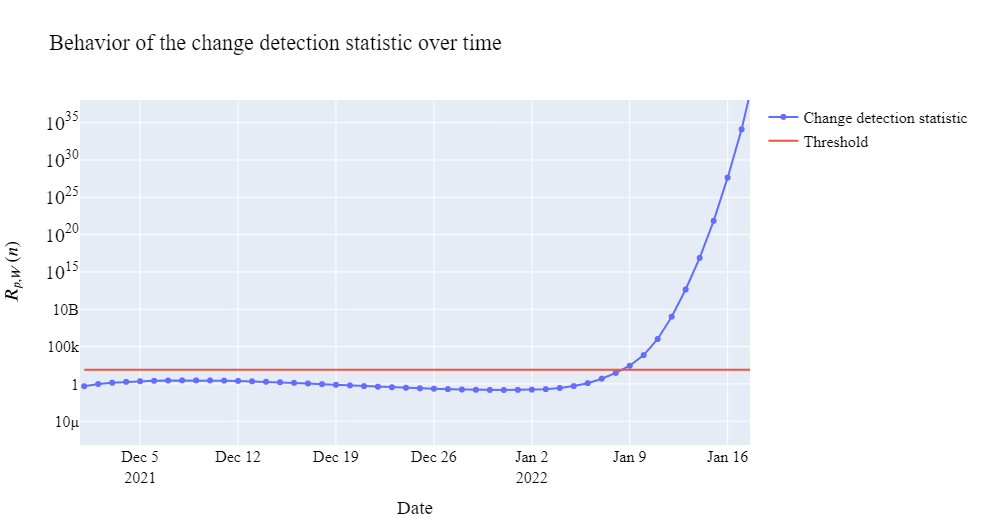}
\caption{Behavior of the change detection statistic $R_{\pb,W}(n)$ over time.}  
\label{Detection_stat}
\end{figure}

The obtained results complement the existing work on the application of changepoint detection algorithms to epidemic detection problems (see, e.g.,
\cite{Baron2004,LiangTarVeerIEEEIT2023,pergamenchtchikov2022minimax,Baron2013}). 
The results show that the proposed mixture-based  change detection algorithm can be useful for governments when deciding whether to impose a total lockdown 
across the country without regard to a region in the early stages of a pandemic.

\subsection{Application to Detection and Extraction of Faint Space Objects} \label{ssec:Space}

The problem of rapid detection and extraction of streaks of low-observable space objects with unknown orbits in optical images captured with ground-based telescopes is a challenge for Space Informatics. A typical image (digital frame) with a low-contrast streak with the signal-to-noise (SNR) 
in pixel approximately $1$ is shown in Fig.~\ref{fig:fit_track}.
%

\begin{figure}[!h!] 
\centering
	\includegraphics[height=0.4\textheight, width = 0.8\textwidth]{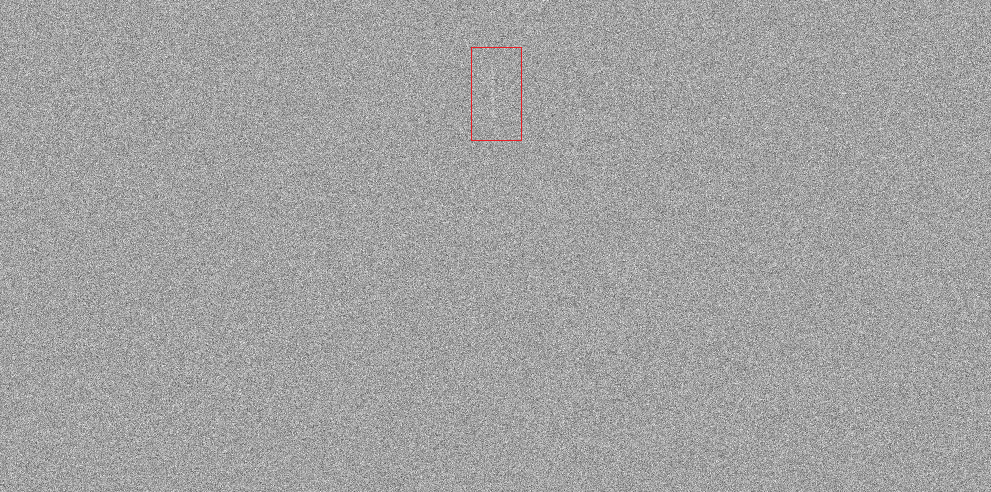}
	\caption{Digital image with a low-SNR streak. The red rectangle marks the streak position.}
	\label{fig:fit_track}
\end{figure}

Since the distribution of observations changes abruptly when the streak starts and ends the problem of object streak extraction can be regarded as 
 the changepoint detection problem in 2-D space (but not in time since we consider a single image).

In realistic situation, the problem is aggravated by the presence of stars and background that produce strong clutter, and 
special image preprocessing for clutter removal has to be implemented (see, e.g., \cite{Tartakovsky&Brown-IEEEAES08}). 
We assume that after appropriate preprocessing, the clutter-removed input frame contains Gaussian 
noise independent from pixel to pixel. For the sake of simplicity, we also consider a scenario where only one streak may be present in the image.
We further assume that the satellite has a linear and uniform motion in the frame. The satellite streak is given by the vector 
$\Yb=(x_0,y_0,x_1,y_1)$, where $(x_0,y_0)$ corresponds to the start point and $(x_1, y_1)$ corresponds to the endpoint. We consider the 
following model of the observation $X_{i,j}$ in pixel $(i,j)$ of the 2-D frame:
\begin{equation}\label{eq:signal_main}
X_{i,j}= \theta S_{i,j}(\Yb) + \xi_{i,j},
\end{equation}
where $\theta$ is an unknown signal intensity from the object,
$\{S_{i,j}(\Yb)\}$ are values of the model profile of the streak that are calculated beforehand assuming the point spread function (PSF) is 
Gaussian with a certain effective width, which is shown in  Fig.~\ref{fig:signal}; and
$\xi_{i,j}\sim \Nc(0,\sigma^2)$ is Gaussian noise after preprocessing with zero mean and known (estimated empirically) local variance $\sigma^2$. 
Thus, the observation $X_{i,j}$ has normal pre-change distribution $g(X_{i,j}) \sim \Nc(0,\sigma^2)$
when the streak does not cover pixel $(i,j)$ and normal post-change distribution $f(X_{i,j}) \sim \Nc(\theta S_{i,j}(\Yb),\sigma^2)$ when the streak 
covers the pixel $(i,j)$.

\begin{figure}[!h!] 
\centering
	\includegraphics[width=\textwidth]{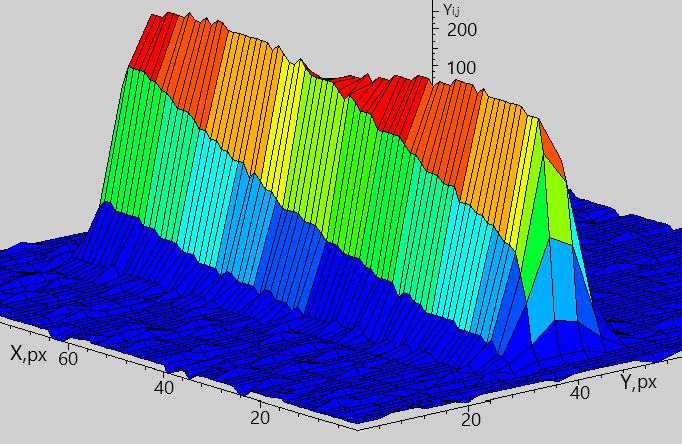}
	\caption{Model profile of the streak.} 
	\label{fig:signal}
	
\end{figure}

The problem is to detect the streak with minimal delay or to make a decision that there is no streak in the frame.

We consider only intra-frame detection of faint streaks of subequatorial satellites with unknown orbits with telescopes 
mounted at the equator. In this case, a signal from a satellite (with unknown intensity, start, and end points) is located almost vertically 
in a small area at the center of the frame, as shown in Fig.~\ref{fig:search_area}.

\begin{figure}[!h] 
\centering
	\includegraphics[width = \textwidth]{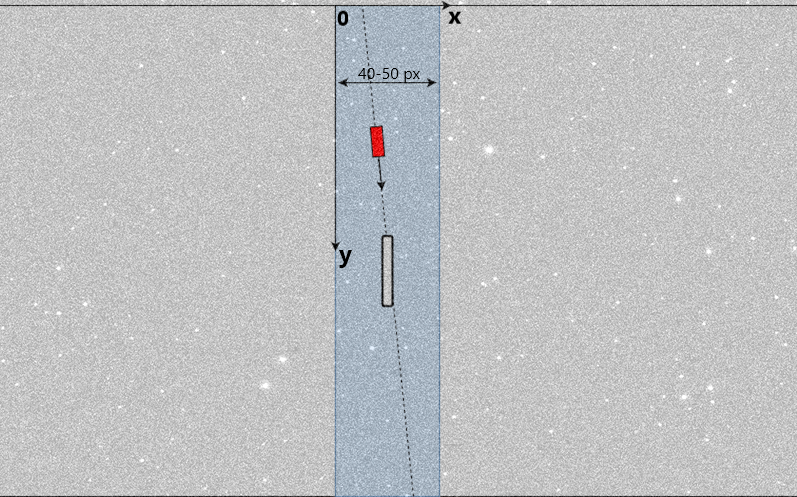}
	\caption{Search area $\Omega_S$ is rectangular. The dotted line shows one of the possible directions. White rectangle -- streak, red rectangle -- sliding window.}
	\label{fig:search_area}
\end{figure}

Let $\Omega_S$ denote the streak search area. We define $N$ different directions 
inside $\Omega_S$. Let $d\in \{1,\dots,N\}$ denote a certain direction. Since we assume that there may appear only one
satellite's streak in the frame there is no need to mix SR statistics over streak location but rather maximize. However, since the signal intensity $\theta$ 
is not known we still need to mix over the distribution of $\theta$, which is a drawback. 

A reasonable way to avoid this averaging is to use the so-called 
Finite Moving Average (FMA) statistics, as suggested in \cite{TartakovskyetalIEEESP2021}. 
Specifically,  let $M_d(n),n \ge 1$  stand for a 2-D sliding rectangular window which 
contains certain pixel numbers $(i,j)$ at each step $n$ in the direction $d$. Window $M_d(n)$ has a fixed length of $L_d$ pixels and a fixed width 
of $K_d$ pixels (the choice of the parameters depends on the expected SNR and PSF effective width).
For the certain direction $d$ ($d=1,\dots,N$), the FMA statistic is defined as
\[
V_{M_d(n)}(n) = \sum_{(i,j)\in M_d(n)}S_{i,j}^{(d)} X_{i,j}^{(d)},
\]
where $\{S_{i,j}^{(d)}\}$ are values of the Gaussian model profile in the direction $d$ and $X_{i,j}^{(d)}$ are observed data in the direction $d$. 
Profile location is given by the vector $\Yb_d$ in the direction $d$. Then the multistream\footnote{Here ``streams'' are not streams per se but rather 
data $X_{i,j}^{(d)}$ in different directions $d \in \{1,\dots,N\}$ in the search area $\Omega_S$.} FMA detection procedure is defined as
\[
T_h=\inf\left\{n \ge 1: \max_{1\le d \le N}V_{M_d(n)}(n) \ge h \right\}.
\]

For the Gaussian model, the FMA procedure is invariant to the unknown signal intensity $\theta$, which is a big advantage over the SR-type
versions.

\begin{figure}[!h!] 
\centering
	\includegraphics[width = \textwidth]{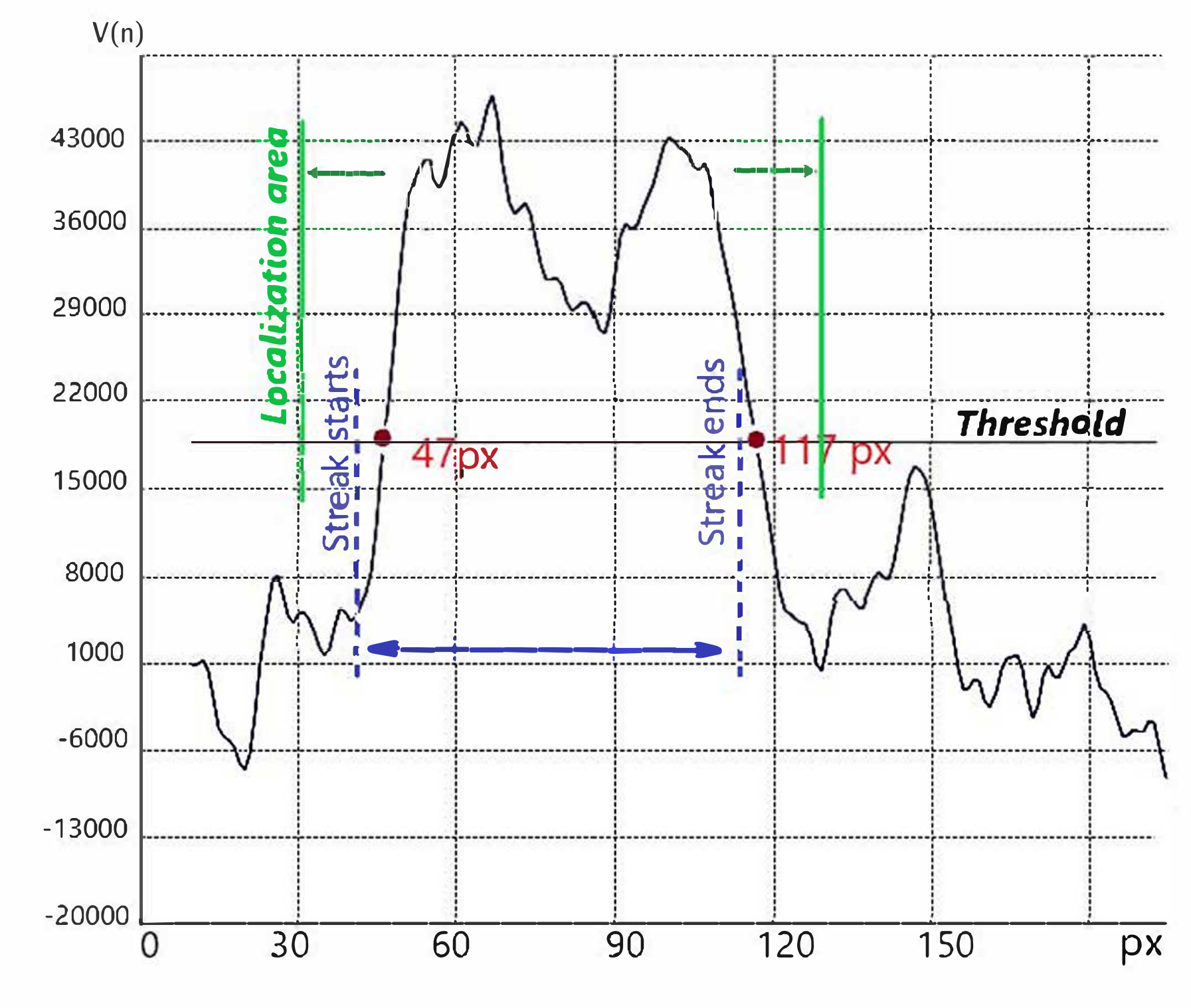}
	\caption{ The behavior of the FMA $V_{M_d(n)}(n)$ statistic along the correct direction.}
	\label{fig:plot_stat_R}
\end{figure}

Fig. \ref{fig:plot_stat_R} shows a typical behavior of the $V_{M_d(n)}(n)$ statistic along the 
direction containing the streak in the case of a very low SNR $ = 0.9$. In this case, the streak is detected with coordinates 
of start and end at points 47 and 117, respectively, while the true values are 40 and 110 so that the precision is $7$ pixels.
Experiments show that when sliding the 2-D window in various directions inside $\Omega_S$ 
and then comparing the largest value $\max_{d} V_{M_d(n)}(n)$ to a threshold we typically determine the approximate position 
of the streak with an accuracy of 5-10 pixels. Therefore, the proposed FMA version of the change detection algorithm turns out to be efficient
 -- it allows us to rapidly determine a localization area, which with a high probability contains the streak.

\section{Concluding Remarks and Future Challenges}\label{sec:Remarks} 

\subsection{Remarks}

1.  As we already discussed, for general non-i.i.d.\ models SR-type statistics cannot be computed recursively even in separate streams, so
 the computational complexity and memory requirements of the mixture detection procedures $T_A^{\pb,\theta}$ and $T_A^{\pb,W}$, 
 especially in the asymptotically non-stationary case,  can be quite high. To avoid this drawback it is reasonable to use either one-stage delayed adaptive procedures
 or window-limited versions of mixture detection procedures where the summation over potential change points $\nu=k$ is restricted to the 
 a sliding window of a fixed size $\ell$. In the window-limited versions of $T_A^{\pb,W}$, the statistic $R_{\pb,W}(n)$ is replaced by the
 window-limited statistic
\begin{align*}
\widehat{R}_{\pb,W}(n) =  \sum_{k=n-(\ell+1)}^{n-1}  \Lambda_{\pb,W}(k,n) \quad \text{for}~n > \ell; \quad \widehat{R}_{\pb,W}(n)= 
R_{\pb,W}(n)  \quad \text{for}~n \le \ell.
\end{align*}
Following guidelines of Lai~\cite{LaiIEEE98} and the techniques developed by Tartakovsky~\cite[Sec 3.10, pages 116-122]{Tartakovsky_book2020} 
for the single-stream scenario, it can be shown that the window-limited version of the SR mixture 
also has first-order asymptotic optimality properties as long as the size of the  window $\ell(A)$ approaches infinity as $A\to\infty$ with the rate
$\ell(A)/\log A \to \infty$.  Since thresholds, $A=A_\alpha$, in detection procedures should be selected in such a way that 
$\log A_\alpha \sim |\log \alpha|$ as $\alpha\to0$, the value of the window size $\ell(\alpha)$ should satisfy
$\lim_{\alpha\to0} [\ell(\alpha)/|\log \alpha|] =\infty$.

2. Similar asymptotic optimality results can be obtained for the mixture CUSUM procedure based on thresholding of the sum of generalized LR statistic
\[
W(n) = \sum_{i=1}^N \max_{ 0\le  k < n} \sup_{\theta_i\in \Theta_i} \lambda_{i,\theta_i}(k,n)
\]
and for the corresponding window-limited version.

3.  We also conjecture that asymptotic optimality properties hold for the multistream Finite Moving Average (FMA) 
detection procedure given by the stopping time
\[
T_h = \inf\set{n \ge 1: \sum_{i=1}^N \sup_{\theta_i\in \Theta_i}\sum_{k=\max(1,n-\ell+1)}^n\lambda^*_{i, \theta_i}(k) \ge h},
\]
where $\lambda^*_{i, \theta_i}(k)= \log [f_{i,\theta_i}(X_k(i)|\Xb^{k-1}(i))/g_{i}(X_k(i)|\Xb^{k-1}(i))]$. In cases where the LLR $\lambda^*_{i, \theta_i}(k)$
is a monotone function of some statistic, the FMA procedure can be appropriately modified in such a way that
it is invariant to the unknown parameters $\theta_i$. This is a great advantage over corresponding SR-based and CUSUM-based procedures.
See also discussion in Subsection~\ref{ssec:Space}.

4. The results can be easily generalized to the case where the change points $\nu_i$ are different for different streams $i=1,\dots,N$.

5. The Bayesian-type results of this paper can be used to establish asymptotic optimality properties of the detection procedures $T_A^{\pb,\theta}$
and $T_A^{\pb,W}$ in a non-Bayesian setting. In particular, the optimization problem can be solved in the class of procedures with the upper-bounded maximal local 
probability of false alarms
\[
\sup_{m \ge 0} \Pb_\infty\brc{T \le m +k|T>m} \le \alpha,
\]
both in pointwise and minimax settings, by embedding into the Bayesian class, similar to what was done by 
Pergamenchtchikov and Tartakovsky~\cite{PerTar-JMVA2019} for the case of a single stream.

\subsection{Future Challenges}

1. The results of MC simulations in Section~\ref{sec:MC} show that first-order approximations to EDD and PFA are typically not especially accurate, 
so higher-order approximations are in order. However, it is not feasible to 
obtain such approximations in the general non-i.i.d.\ case considered in this paper. Higher order approximations to the expected detection delay 
and the probability of false alarm for the i.i.d. models, assuming that the observations in streams are independent and also independent across streams, 
can be derived based on the renewal and nonlinear renewal theories. This important problem will be considered in the future.

2. The results of this paper cover the scenario where the number of streams $N$ is fixed so that $\log(N/\alpha)\sim \log(1/\alpha)$ for small $\alpha$. 
In Big Data problems, $N$ may be very large and go to infinity. These problems require different approaches and different detection procedures. 
This challenging problem will be considered in the future. 

3.  For detecting transient (or intermittent) changes of unknown duration (like object streaks in Subsection~\ref{ssec:Space}) it is often more reasonable to consider
not quickest detection criteria but reliable detection criteria that require minimization of the detection probability in a fixed time (or space) window (see, e.g., 
\cite{BerenkovTarKol_EnT2020,TartakovskyetalIEEESP2021,Tartakovsky_book2020} and references therein). While certain interesting asymptotic results 
for single-stream scenarios and i.i.d.\ data models exist
\cite{Nikiforov+et+al:2012,Nikiforov+et+al:2017,Nikiforov+et+al:2023,SokolovSpivakTartakSQA2023,TartakovskyetalIEEESP2021,Tartakovsky_book2020}, 
to the best of our knowledge this problem has never been considered in the multistream setting and for non-i.i.d.\ models. 

\begin{acknowledgement}
We are grateful to a referee for a comprehensive review and useful comments.
\end{acknowledgement}


\end{document}